\documentclass[a4paper,11pt]{amsart}
\usepackage{amssymb}
\usepackage{amsmath}
\usepackage{hyperref}
\usepackage{fullpage}

\DeclareMathOperator{\Rep}{Re}

\DeclareMathOperator{\Li}{li}
\DeclareMathOperator{\Ei}{Ei}
\DeclareMathOperator{\Res}{Res}

\numberwithin{equation}{section}

\theoremstyle{plain}
\newtheorem{Theorem}{Theorem}
\newtheorem{Cor}[Theorem]{Corollary}
\newtheorem{Lemma}{Lemma}[section]
\newtheorem{Proposition}[Lemma]{Proposition}

\newtheorem{Conjecture}[Lemma]{Conjecture}

\theoremstyle{definition}
\newtheorem{Definition}[Lemma]{Definition}
\newtheorem{Remark}[Lemma]{Remark}

\newcommand{\R}{\mathbb {R}}
\newcommand{\C}{\mathbb {C}}
\newcommand{\Z}{\mathbb {Z}}
\newcommand{\Q}{\mathbb {Q}}
\newcommand{\ac}{\sigma_{\rm c}}
\newcommand{\wF}{\widetilde{F}}
\newcommand{\wFp}{\widetilde{F_{+}}}

\newcommand{\EC}{c_{\rm E}}

\begin{document}
\title[Divisor functions and zeros of the zeta-function]
{
Maximal order for divisor functions
and zeros of the Riemann zeta-function
}
\author[H.~Akatsuka]{Hirotaka Akatsuka}
\address{%
Otaru University of Commerce,
3--5--21, Midori, Otaru, Hokkaido, 047--8501,
Japan.
}
\email{akatsuka@res.otaru-uc.ac.jp}

\begin{abstract}
Ramanujan investigated maximal order for the number of divisors
function by introducing some notion such as
(superior) highly composite numbers.
He also studied maximal order for other arithmetic functions
including the sum of powers of divisors function.
In this paper we relate zero-free regions for the Riemann zeta-function
to maximal order for the sum of powers of divisors function.
In particular, we give equivalent conditions for the Riemann hypothesis
in terms of the sum of the $1/2$-th powers of divisors function.
As a by-product,
we also give equivalent conditions for the Riemann hypothesis in terms of
the partial Euler product for the Riemann zeta-function.
\end{abstract}
\maketitle
\section{Introduction}\label{Sec:Intro}
Let $a(n)$ be a positive-valued arithmetic function.
For a continuous, monotonically increasing function $f:[1,\infty)\to\R_{>0}$
we say that $f$ is a \emph{maximal order} for $a(n)$ if
\[
 \limsup_{n\to\infty}\frac{a(n)}{f(n)}=1.
\] 
The maximal order is a notion to look for the best possible upper bounds
for arithmetic functions in terms of continuous functions.
See \cite[\S1.5]{Te} for an introductory exposition.

For $\kappa\in\R$ we define a divisor function $\sigma_{\kappa}(n)$
by
\begin{equation}\label{eq:defdivft}
 \sigma_{\kappa}(n)=\sum_{d\mid n}d^{\kappa},
\end{equation}
where $d$ runs over the positive divisors of $n$.
We also set
\begin{equation}\label{eq:defdivft2}
 d(n)=\sigma_0(n)=\sum_{d\mid n}1,\phantom{MM}
\sigma(n)=\sigma_1(n)=\sum_{d\mid n}d.
\end{equation}
In \cite[eqs. (12) and (25)]{Gr} Gronwall gave maximal order for $\sigma_{\kappa}(n)$
when $\kappa\geq 1$.
In fact, he proved
\[
 \limsup_{n\to\infty}\frac{\sigma(n)}{n\log\log n}=e^{\EC},\phantom{MM}
\limsup_{n\to\infty}\frac{\sigma_{\kappa}(n)}{n^{\kappa}}=\zeta(\kappa)
\]
for $\kappa>1$,
where $\EC$ is the Euler--Mascheroni constant
and $\zeta(s)$ is the Riemann zeta-function.
For $0\leq\kappa<1$ we know maximal order for $\log(\sigma_{\kappa}(n)/n^{\kappa})=\log(\sigma_{-\kappa}(n))$ unconditionally,
instead of $\sigma_{\kappa}(n)$ itself.
Namely, in \cite{Wi} Wigert obtained
\[
 \limsup_{n\to\infty}\frac{\log d(n)}{\log n/\log\log n}=\log 2
\]
and in \cite[the last formula]{Gr} Gronwall showed
\[
 \limsup_{n\to\infty}\frac{\log(\sigma_{\kappa}(n)/n^{\kappa})}{(\log n)^{1-\kappa}/\log\log n}=\frac{1}{1-\kappa}
\]
for $\kappa\in(0,1)$.

In \cite{Ra1} Ramanujan investigated large values of $d(n)$ systematically.
To construct explicit sequences
of positive integers
on which $d(n)$ takes large values,
he introduced various notion such as highly composite numbers
and superior highly composite numbers.
See \cite{Ni} for a survey on \cite{Ra1} and related progress until 1980s.
See also \cite[Chapter 10]{MM} for another survey.
As was mentioned in \cite[Notes on the paper No.15 in pp.338--339]{Ra3},
a part of the paper \cite{Ra1} was suppressed.
The suppressed part was published in \cite{Ra4} with an annotation
by Nicolas and Robin.\footnote{%
We can also find this part with revised and extended commentary in
\cite[Chapter 10]{AB}.
A facsimile version of the original manuscript is available
in \cite[pp.280--312]{Ra2}.
}
In this part, maximal order for various arithmetic functions was
treated.
In \cite[\S 57--\S 71]{Ra4} he discussed large values
for $\sigma_{\kappa}(n)$ and deduced
\begin{equation}\label{eq:Rama}
 \limsup_{n\to\infty}a_{\kappa}(n)
=\begin{cases}
  -\frac{1}{\sqrt{2}}\zeta(1/2) & \text{if $\kappa=1/2$,}\\
      -\zeta(\kappa) & \text{if $1/2<\kappa<1$}
 \end{cases}
\end{equation}
under the Riemann hypothesis,
where $a_{\kappa}(n)$ is defined by
\begin{equation}\label{eq:defak}
 a_{\kappa}(n)=\frac{\sigma_{\kappa}(n)}{n^{\kappa}\exp[\Li((\log n)^{1-\kappa})]}
\end{equation}
and $\Li(x)$ is the logarithmic integral given by
\begin{equation}\label{eq:defLi}
 \Li(x)=\lim_{\delta\downarrow 0}
\left(\int_0^{1-\delta}+\int_{1+\delta}^x\right)\frac{du}{\log u}.
\end{equation}

Maximal order for divisor functions is sometimes related to
the distribution of zeros of the Riemann zeta-function.
The typical example is Robin's criterion for the Riemann hypothesis
\cite{Ro1,Ro2},
which states that the Riemann hypothesis is true if and only if
the inequality $\sigma(n)<e^{\EC}n\log\log n$ holds for any
$n>5040$.
See \cite{NS} for historical accounts on
Robin's result as well as Ramanujan's work.
We also refer to \cite{Lag} for a reformulation of
Robin's criterion in terms of the harmonic numbers instead of
the Euler--Mascheroni constant.
See \cite[Chapter 7]{Br1} for an exposition on the works of
Robin and Lagarias.

The aim of this paper is to discuss relations between (\ref{eq:Rama})
and the distribution of zeros of $\zeta(s)$.
The main result is stated as follows:
\begin{Theorem}\label{Thm1}
 Let $\kappa\in[\frac{1}{2},1)$.
Then the following statements {\rm (i)}--{\rm (iii)}
are equivalent:
\begin{itemize}
\item[(i)] There are no zeros of  $\zeta(s)$ in $\Rep(s)>\kappa$.
\item[(ii)] The function
$a_{\kappa}(n)$ is bounded on $n\in\Z_{\geq 3}$.
\item[(iii)] The formula {\rm (\ref{eq:Rama})} holds.
\end{itemize}
\end{Theorem}
Specializing $\kappa=1/2$ in Theorem \ref{Thm1},
we immediately obtain equivalent conditions for the Riemann
hypothesis. Namely,
\begin{Cor}\label{Cor2}
 The following statements {\rm (i)}--{\rm (iii)} are equivalent:
\begin{itemize}
 \item[(i)] The Riemann hypothesis is true.
 \item[(ii)] The function $a_{1/2}(n)$
is bounded on $n\in\Z_{\geq 3}$.
\item[(iii)] We have $\limsup_{n\to\infty}a_{1/2}(n)=-\zeta(\tfrac{1}{2})/\sqrt{2}$.
\end{itemize}
\end{Cor}
While Robin's criterion for the Riemann hypothesis
is an inequality between $\sigma(n)$
and its maximal order, Corollary \ref{Cor2} characterizes
the Riemann hypothesis
in terms of maximal order for $\sigma_{1/2}(n)$ itself.

The implication (iii)$\Longrightarrow$(ii) in Theorem \ref{Thm1}
is trivial.
So this paper is devoted to proving the implication
(ii)$\Longrightarrow$(i) and (i)$\Longrightarrow$(iii).
As has been explained above, Ramanujan derived the implication
(i)$\Longrightarrow$(iii) in Corollary \ref{Cor2}.

We explain
our strategy for deriving Theorem \ref{Thm1}.
First of all we develop $\kappa$-superior highly composite numbers.
This produces a sequence of positive integers $n$
on which $\sigma_{\kappa}(n)$ takes large values. 
Next we connect $\sigma_{\kappa}(n)$ on $\kappa$-superior highly
composite numbers with the partial Euler product
for $\zeta(s)$ at $s=\kappa$ (see Proposition \ref{PropEP}).
After that, we consider the behavior of the partial Euler product,
dividing the cases whether there is a zero of $\zeta(s)$
in $\Rep(s)>\kappa$ or not.
In the case where $\zeta(s)$ has no zeros in $\Rep(s)>\kappa$,
we investigate it by establishing a kind of explicit formulas
(see Proposition \ref{PropEP2}).
As a result, we will obtain the implication (i)$\Longrightarrow$(iii).
On the other hand, in the case where there is a zero of $\zeta(s)$
in $\Rep(s)>\kappa$, we show an $\Omega$-estimate for the partial
Euler product (see Proposition \ref{Prop:EPomega}).
This leads to the implication (ii)$\Longrightarrow$(i).

As a by-product of this paper, we characterize the
Riemann hypothesis in terms of the partial Euler product
for $\zeta(s)$ at $s=1/2$.
To state the result, we set
\begin{equation}\label{eq:E1}
 E_1(X)=\left.\prod_{p\leq X}(1-p^{-1/2})^{-1}\right/\exp[\Li(\vartheta(X)^{1/2})],
\end{equation}
where $p$ runs over the prime numbers up to $X$ and $\vartheta(X)=\sum_{p\leq X}\log p$.
Then we have
\begin{Theorem}\label{Thm3}
 The following statements {\rm (i)}--{\rm (iii)} are equivalent:
\begin{itemize}
 \item[(i)] The Riemann hypothesis is true.
 \item[(ii)] The function $E_1(X)$ is bounded on $X\geq 3$.
 \item[(iii)] The function $E_1(X)$ converges to $-\sqrt{2}\zeta(1/2)$
as $X\to\infty$.
\end{itemize}
\end{Theorem}
\begin{Remark}
The implication (i)$\Longrightarrow$(iii) has been discovered by
Ramanujan (see \cite[eq.~(359)]{Ra4}).

As was obtained in \cite[Corollary 4.4]{Ak},
the condition that
\[
 E_2(X)=\left.\prod_{p\leq X}(1-p^{-1/2})^{-1}\right/\exp[\Li(X^{1/2})]
\]
converges to $-\sqrt{2}\zeta(1/2)$ is equivalent to $\psi(X)=X+o(X^{1/2}\log X)$
as $X\to\infty$, where $\psi(X)=\sum_{n\leq X}\Lambda(n)$ and $\Lambda(n)$
is the von Mangoldt function.
We note that the Riemann hypothesis is equivalent to $\psi(X)=X+O(X^{1/2}(\log X)^2)$
and that the error estimate $O(X^{1/2}(\log X)^2)$ is the best bound under
the Riemann hypothesis at present.
From this viewpoint, the convergence of $E_1(X)$ is a little easier
than that of $E_2(X)$.

On the convergence of $E_1(X)$ and $E_2(X)$,
the situation changes due to the difference between $X$ and $\vartheta(X)$.
In this regard
we refer to the following  criterion
for the Riemann hypothesis given by Robin and Nicolas.
We recall the classical result by Littlewood \cite{Li} (see also \cite[Theorem 15.11 in p.479]{MV})
that $\Li(X)-\pi(X)$ changes sign infinity often,
where $\pi(X)=\#\{p\leq X:\text{primes}\}$.
On the other hand, Robin \cite[Theoreme 1]{Ro3}
obtained that the condition
$\Li(\vartheta(X))-\pi(X)>0$ for any large $X$ is equivalent to
the Riemann hypothesis.
In \cite[Corollary 1.2]{Ni2}
Nicolas showed that the phrase ``for any large $X$'' can be replaced
by ``for any $X\geq 11$''.
See also \cite[Chapter 1]{Br2} for this criterion.
\end{Remark}

This paper is organized as follows.
In \S \ref{Sec:SHCN} we define $\kappa$-superior highly composite
numbers and develop their properties.
In \S \ref{Sec:sigSHCN1} we investigate the behavior of $\sigma_{\kappa}(n)$ on
$\kappa$-superior highly composite numbers.
In particular, we connect it to the partial Euler product
for $\zeta(s)$ at $s=\kappa$.
In \S \ref{Sec:EP} we express the partial Euler product in terms of
sums over zeros of the Riemann zeta-function, which can be
regarded as a kind of explicit formulas.
In \S \ref{Sec:sigSHCN2} we rewrite a result of \S \ref{Sec:sigSHCN1}
when $1/2\leq\kappa<1$.
In \S \ref{Sec:pf1} we prove the condition (iii) under
the condition (i) in Theorem \ref{Thm1},
using the explicit formula derived in \S \ref{Sec:EP}.
In \S \ref{Sec:pf2} we firstly develop $\Omega$-estimates
for a partial Dirichlet
series related to the partial Euler product if the condition (i)
is false in Theorem \ref{Thm1}.
Using this, we derive the implication (ii)$\Longrightarrow$(i)
by showing that $a_{\kappa}(n)$ is not bounded if the condition
(i) is false.
In \S \ref{Sec:pf3} we prove Theorem \ref{Thm3}.
\subsection*{Notation and Conventions}
Throughout this paper we will use the following notation and
conventions.
\begin{itemize}
\item The letter $p$ always denotes a prime number.
\item Let $\mathcal{S}$ be a set. Let $f$ be a complex-valued function
and $g$ be a positive valued function on $\mathcal{S}$.
We write $f(x)=O(g(x))$ or $f(x)\ll g(x)$
on $x\in\mathcal{S}$ if there exists $C>0$ such that
$|f(x)|\leq C g(x)$ on $x\in\mathcal{S}$.
If the constant $C$ may depend on parameters $\alpha$, $\beta,\ldots$,
then we write $f(x)=O_{\alpha,\beta,\ldots}(g(x))$ or
$f(x)\ll_{\alpha,\beta,\ldots}g(x)$ on $x\in\mathcal{S}$.
\item Let $X_0\in\R$.
Let $f$ be a real-valued function and $g$ be a positive-valued function
on the interval $[X_0,\infty)$.
We write $f(x)=\Omega_{+}(g(x))$ as $x\to\infty$
if $\limsup_{x\to\infty}f(x)/g(x)>0$.
\end{itemize}
We continue to use the notation explained in the Introduction. 
Namely,
$\sigma_{\kappa}(n)$, $d(n)$ and $\sigma(n)$
are the divisor functions defined by
(\ref{eq:defdivft}) and (\ref{eq:defdivft2}).
The arithmetic function $\Lambda(n)$ is the von Mangoldt function
and we set $\psi(x)=\sum_{n\leq x}\Lambda(n)$ as usual.
The function $\vartheta(x)=\sum_{p\leq x}\log p$ is the Chebyshev function
and $\pi(x)=\#\{p\leq x\colon\text{primes}\}$ is the prime counting function.
The number $\EC$ is the Euler--Mascheroni constant.
The function $\zeta(s)$ is the Riemann zeta-function.
The function $\Li(x)$ is the logarithmic integral defined by (\ref{eq:defLi}).
We reserve the notation $a_{\kappa}(n)$ and $E_1(X)$
as in (\ref{eq:defak}) and (\ref{eq:E1}), respectively.
Other notation will be introduced as needed.
\subsection*{Acknowledgments}
The author was supported by JSPS KAKENHI Grant Number JP1903392.
\section{$\kappa$-superior highly composite numbers}\label{Sec:SHCN}
We take $\kappa>0$ and fix it throughout this section
unless otherwise specified.
First of all we introduce a $\kappa$-superior highly composite number.
\begin{Definition}\label{def:kSHCN}
Let $\varepsilon>0$.
A positive integer $N$ is
a \emph{$\kappa$-superior highly composite number with parameter $\varepsilon$}
(\emph{$\kappa$-SHCN with parameter $\varepsilon$}, for short)
if
\begin{equation}\label{eq:kappaCA}
 \frac{\sigma_{\kappa}(n)}{n^{\kappa(1+\varepsilon)}}
\leq
\frac{\sigma_{\kappa}(N)}{N^{\kappa(1+\varepsilon)}}
\end{equation}
holds for any $n\in\Z_{\geq 1}$.
\end{Definition}
\begin{Remark}
The notion of $\kappa$-SHCNs is a variant or a generalization
of the following notions:
\begin{itemize}
\item (Superior highly composite numbers \cite[A002201]{Sl})
A positive integer $N$ is a \emph{superior highly composite number}
(\emph{SHCN}, for short)
if there exists
$\varepsilon>0$ such that $n^{-\varepsilon}d(n)\leq N^{-\varepsilon}d(N)$
holds for any $n\in\Z_{\geq 1}$.
This notion was essentially introduced by Ramanujan \cite[Chapter IV]{Ra1}
though the original definition is slightly different.
See also \cite[\S IV]{Ni}.
\item (Colossally abundant numbers \cite[A004490]{Sl})
A positive integer $N$ is \emph{colossally abundant} (\emph{CA}, for short)
if there exists $\varepsilon>0$ such that
$n^{-1-\varepsilon}\sigma(n)\leq N^{-1-\varepsilon}\sigma(N)$
holds for any $n\in\Z_{\geq 1}$.
This was essentially introduced by Alaoglu and Erd\H{o}s \cite[\S 3]{AE}
and modified as the above form by Erd\H{o}s and Nicolas \cite[\S 2]{EN}.
\end{itemize}

In \cite[\S 59]{Ra4} Ramanujan introduced the essentially same notion of
$\kappa$-SHCNs and he called them
\emph{generalised superior highly composite numbers}.
In this paper we say $\kappa$-SHCNs in order to clarify
the dependency of $\kappa$.

For notions related to SHCNs and CA numbers,
see \cite{AE,CNS,EN,NY,Ra1,Ra4}.
See also \cite[Chapter 6]{Br1} for their exposition.
\end{Remark}
We fix $\varepsilon>0$.
Since $\sigma_{\kappa}(n)\leq n^{\kappa}d(n)$,
we see that $n^{-\kappa(1+\varepsilon)}\sigma_{\kappa}(n)$
tends to $0$ as $n\to\infty$.
Thus $n^{-\kappa(1+\varepsilon)}\sigma_{\kappa}(n)\leq 1/2$
holds for sufficiently large $n$.
This implies that there exists $N=N(\varepsilon)$
such that (\ref{eq:kappaCA}) holds for any $n\in\Z_{\geq 1}$.

Next we determine $\kappa$-SHCNs with fixed parameter $\varepsilon>0$.
We note that $\sigma_{\kappa}(n)$ is a multiplicative function.
Therefore, when we factorize $n=\prod_p{p^{e_p}}$,
we have
\[
 \frac{\sigma_{\kappa}(n)}{n^{\kappa(1+\varepsilon)}}
=\prod_{p}\frac{\sigma_{\kappa}(p^{e_p})}{p^{\kappa e_p(1+\varepsilon)}}.
\]
Hence, for each prime number $p$
we look for $E=E_{p,\varepsilon}\in\Z_{\geq 0}$ such that
\begin{equation}\label{eq:pexp}
 \frac{\sigma_{\kappa}(p^e)}{p^{\kappa e(1+\varepsilon)}}
\leq
\frac{\sigma_{\kappa}(p^E)}{p^{\kappa E(1+\varepsilon)}}
\end{equation}
holds for any $e\in\Z_{\geq 0}$.
Such an $E$ exists because $p^{-\kappa e(1+\varepsilon)}\sigma_{\kappa}(p^e)$
tends to $0$ as $e\to\infty$.
Below we give explicit expressions for $E$.
\begin{Lemma}\label{Lempexp1}
 Let $p$ be a prime number and $\varepsilon>0$.
Take $E\in\Z_{\geq 0}$ such that
{\rm (\ref{eq:pexp})} holds for any $e\in\Z_{\geq 0}$.
Then we have $\lambda_p(\varepsilon)-1\leq E\leq\lambda_p(\varepsilon)$,
where
\begin{equation}\label{eq:lambda}
 \lambda_p(\varepsilon)=
\left.\log
   \left(
      \frac{p^{\kappa(1+\varepsilon)}-1}{p^{\kappa}(p^{\kappa\varepsilon}-1)}
   \right)
\right/
\log(p^{\kappa}).
\end{equation}
\end{Lemma}
\begin{proof}
Putting $e=E+1$ on (\ref{eq:pexp}), we have
$\sigma_{\kappa}(p^{E+1})\leq p^{\kappa(1+\varepsilon)}\sigma_{\kappa}(p^E)$.
Using $\sigma_{\kappa}(p^e)=(p^{\kappa(e+1)}-1)/(p^{\kappa}-1)$,
we easily find $E\geq\lambda_p(\varepsilon)-1$.
Next we show $E\leq\lambda_p(\varepsilon)$.
When $E=0$, this inequality is automatically true because
$(p^{\kappa(1+\varepsilon)}-1)/\left(p^{\kappa}(p^{\kappa\varepsilon}-1)\right)>1$.
When $E\geq 1$, we obtain the inequality by putting $e=E-1$
on (\ref{eq:pexp}).
This completes the proof.
\end{proof}
\begin{Proposition}\label{Proppexp2}
Keep the notation in Lemma {\rm \ref{Lempexp1}}.
Then we have
\begin{equation}\label{eq:pexp2}
\begin{aligned}
&\{E\in\Z_{\geq 0}:\text{the inequality }(\ref{eq:pexp})\text{ holds for any }e\in\Z_{\geq 0}\}\\
&=\begin{cases}
   \{\lfloor \lambda_p(\varepsilon)\rfloor\} & \text{if }\lambda_p(\varepsilon)\notin\Z,\\
\{\lambda_p(\varepsilon)-1,\lambda_p(\varepsilon)\} & \text{if }\lambda_p(\varepsilon)\in\Z,
  \end{cases}
\end{aligned}
\end{equation}
where $\lfloor x\rfloor$ is the largest integer not exceeding $x$.
\end{Proposition}
\begin{proof}
 When $\lambda_p(\varepsilon)\notin\Z$, Lemma \ref{Lempexp1} immediately
gives the desired result.
We consider the case $\lambda_p(\varepsilon)\in\Z$.
Using
\begin{equation}\label{eq:plambda}
 p^{\kappa\lambda_p(\varepsilon)}
=\frac{p^{\kappa(1+\varepsilon)}-1}{p^{\kappa}(p^{\kappa\varepsilon}-1)},
\end{equation}
we easily see that
\[
 \frac{\sigma_{\kappa}(p^{\lambda_p(\varepsilon)-1})}
      {p^{\kappa(\lambda_p(\varepsilon)-1)(1+\varepsilon)}}
=\frac{\sigma_{\kappa}(p^{\lambda_p(\varepsilon)})}
      {p^{\kappa\lambda_p(\varepsilon)(1+\varepsilon)}}
=\frac{p^{\kappa(1+2\varepsilon)}(p^{\kappa\varepsilon}-1)^{\varepsilon}}
{(p^{\kappa(1+\varepsilon)}-1)^{1+\varepsilon}}.
\]
Combining Lemma \ref{Lempexp1}, we obtain the result.
\end{proof}
We fix $\varepsilon>0$.
For each prime number $p$ we take an element $E=E_{p,\varepsilon}$
in (\ref{eq:pexp2}).
Since $\lambda_p(\varepsilon)$ tends to $0$ as $p\to\infty$,
we have $E_{p,\varepsilon}=0$ for sufficiently large prime numbers $p$.
Therefore we can define
\begin{equation}\label{eq:kappaCA2}
N=\prod_{p}p^{E_{p,\varepsilon}}.
\end{equation}
Then $N$ is a $\kappa$-SHCN with parameter $\varepsilon$.
Conversely we easily see that any $\kappa$-SHCNs with parameter $\varepsilon$
can be expressed
as (\ref{eq:kappaCA2}).

Next we investigate how the set (\ref{eq:pexp2}) changes
as $\varepsilon$ varies.
For this purpose
we show the following properties of $\lambda_p(\varepsilon)$:
\begin{Lemma}\label{Lemlambda}
Let $p$ be a prime number.
Then the $C^{\infty}$-function $\lambda_p:(0,\infty)\to(0,\infty)$,
which is defined by {\rm (\ref{eq:lambda})},
satisfies the following:
\begin{enumerate}
 \item The function $\lambda_p$ is strictly decreasing, i.e.,
$\lambda_p(\varepsilon)>\lambda_p(\varepsilon')$
if $0<\varepsilon<\varepsilon'$.
\item We have $\lambda_p(\lambda_p(\varepsilon))=\varepsilon$ for
any $\varepsilon\in(0,\infty)$.
\end{enumerate}
\end{Lemma}
\begin{proof}
When $0<\varepsilon<\varepsilon'$,
we have
\[
 \frac{p^{\kappa(1+\varepsilon)}-1}{p^{\kappa\varepsilon}-1}
-\frac{p^{\kappa(1+\varepsilon')}-1}{p^{\kappa\varepsilon'}-1}
=\frac{(p^{\kappa}-1)(p^{\kappa\varepsilon'}-p^{\kappa\varepsilon})}
{(p^{\kappa\varepsilon}-1)(p^{\kappa\varepsilon'}-1)}>0,
\]
which implies the claim (1).
The claim (2) follows from a routine calculation by using
(\ref{eq:plambda}).
\end{proof}
We see from Lemma \ref{Lemlambda} (1) that
$\lambda_p(1)>\lambda_p(2)>\cdots>\lambda_p(n)>\cdots\to 0$.
When $\varepsilon$ varies, the set (\ref{eq:pexp2}) are given as follows:
\begin{Proposition}\label{Proppexp3}
 Keep the notation in Lemma {\rm \ref{Lempexp1}}.
Then we have
\begin{enumerate}
 \item If $\varepsilon>\lambda_p(1)$, then the set {\rm (\ref{eq:pexp2})}
is $\{0\}$.
 \item If there exists $n\in\Z_{\geq 1}$ such that
$\lambda_p(n)>\varepsilon>\lambda_p(n+1)$,
then the set {\rm (\ref{eq:pexp2})} is $\{n\}$.
 \item If $\varepsilon=\lambda_p(n)$ for some $n\in\Z_{\geq 1}$,
then the set {\rm (\ref{eq:pexp2})} is $\{n-1,n\}$.
\end{enumerate}
\end{Proposition}
\begin{proof}
 Firstly we consider the case $\varepsilon>\lambda_p(1)$.
Then by Lemma \ref{Lemlambda} we have
$\lambda_p(\varepsilon)<\lambda_p(\lambda_p(1))=1$.
Combining Lemma \ref{Lempexp1}, we reach the claim (1).

Next we treat the case $\lambda_p(n)>\varepsilon>\lambda_p(n+1)$
for some $n\in\Z_{\geq 1}$.
Thanks to Lemma \ref{Lemlambda} we have $n<\lambda_p(\varepsilon)<n+1$.
By Lemma \ref{Lempexp1} we obtain the claim (2).

Finally we deal with $\varepsilon=\lambda_p(n)$ for some
$n\in\Z_{\geq 1}$.
Then by Lemma \ref{Lemlambda} we have $\lambda_p(\varepsilon)=n$.
By Proposition \ref{Proppexp2} we obtain the claim (3).
\end{proof}
For each prime number $p$ we put $\mathcal{E}_p=\{\lambda_p(n):n\in\Z_{\geq 1}\}$.
We also define $\mathcal{E}$ by
\[
 \mathcal{E}=\bigcup_{p}\mathcal{E}_p.
\]
Next we show that an accumulation point of $\mathcal{E}$ in $\R$ is $0$ only.
Strictly speaking, we prove
\begin{Lemma}\label{Lemaccum}
 For any $\delta>0$ we have
\[
 \#\{(p,n):p\text{ are primes, }n\in\Z_{\geq 1},~\lambda_p(n)\geq\delta\}
<\left(\frac{1}{\kappa\delta\log 2}\right)^{1/\kappa}.
\]
\end{Lemma}
\begin{proof}
Let $p$ be prime numbers and $n$ be positive integers
satisfying $\lambda_p(n)\geq\delta$.
Using $\log(1+x)\leq x$ for $x\geq 0$, we notice that
\[
 \log\left(\frac{p^{\kappa(n+1)}-1}{p^{\kappa}(p^{\kappa n}-1)}\right)
=\log\left(1+\frac{1}{p^{\kappa n}+p^{\kappa(n-1)}+\cdots+p^{\kappa}}\right)
\leq\frac{1}{p^{\kappa n}}.
\]
Thus by $\lambda_p(n)\geq\delta$ we have
 $p^{-\kappa n}\geq\delta\log(p^{\kappa})
\geq\kappa\delta\log 2$.
Therefore we obtain $p^n\leq(1/(\kappa\delta\log 2))^{1/\kappa}$.
This implies the result.
\end{proof}
By Lemma \ref{Lemaccum} we can write
$\mathcal{E}=\{\varepsilon_1,\varepsilon_2,\ldots\}$
with $\varepsilon_1>\varepsilon_2>\cdots\to 0$.
We put $\varepsilon_0=\infty$.
For each $j$ we define
$\mathcal{P}_j=\{p:\text{primes}, \varepsilon_j\in\mathcal{E}_p\}$.
We note that Lemma \ref{Lemaccum} gives
$\#\mathcal{P}_j<(1/(\kappa\varepsilon_j\log 2))^{1/\kappa}<\infty$.
From Proposition \ref{Proppexp3} we immediately see
\begin{Proposition}\label{PropkappaCA}
 Keep the notation as above.
For $\varepsilon\in\R_{>0}$ we set
\begin{equation}\label{eq:Neps}
 N(\varepsilon)=\prod_{p}p^{\lfloor\lambda_p(\varepsilon)\rfloor}.
\end{equation}
Then,
\begin{enumerate}
 \item If $\varepsilon\in\R_{>0}\setminus\mathcal{E}$, then
$N(\varepsilon)$ is the unique $\kappa$-SHCN with parameter $\varepsilon$.
 \item If both $\varepsilon$ and $\varepsilon'$ are in
$(\varepsilon_j,\varepsilon_{j-1})$ for some $j\in\Z_{\geq 1}$,
then $N(\varepsilon)=N(\varepsilon')$.
Below we denote this $\kappa$-SHCN depending only on
$j$ (and $\kappa$) by $N_j$.
\item We have the relation between $N_j$ and $N_{j+1}$:
\[
 N_{j+1}=N_j\times\prod_{p\in \mathcal{P}_j}p.
\]
\item The set of all the $\kappa$-SHCNs
with parameter $\varepsilon_j$ are
\[
\left\{
 N_j\times\prod_{p\in \mathcal{S}}p:
\emptyset\subseteq\mathcal{S}\subseteq \mathcal{P}_j
\right\}.
\]
In particular, $N_j$ and $N_{j+1}(=N(\varepsilon_j))$ are $\kappa$-SHCNs
with parameter $\varepsilon_j$.
\end{enumerate}
\end{Proposition}
(Since the results in the rest of this section are not required for
the latter sections,
readers may proceed to the next section.)
Finally in this section we discuss the number of $\kappa$-SHCNs
with fixed parameter $\varepsilon$.
It follows from Lemma \ref{Lemaccum} and Proposition \ref{PropkappaCA} 
that for any $\kappa\in\R_{>0}$ and $\varepsilon>0$ we have
\[
 \#\{\kappa\text{-SHCNs with parameter }\varepsilon\}
<2^{(1/(\kappa\varepsilon\log 2))^{1/\kappa}}.
\]
On the other hand,
Erd\H{o}s and Nicolas \cite[Proposition 4]{EN} showed
that the number of CA numbers
with fixed parameter $\varepsilon$ is $1$, $2$ or $4$
(see also \cite[the second paragraph after Theorem 10 in p.455]{AE}).
When $\kappa\in\Q_{>0}$, a similar statement holds as follows:
\begin{Proposition}\label{Propnumber}
Let $\kappa\in\Q_{>0}$ and $\varepsilon\in\R_{>0}$.
Then the number of $\kappa$-SHCNs with parameter $\varepsilon$
is $1$, $2$ or $4$.
\end{Proposition}
In order to show this, we quote the six exponentials theorem
(see \cite[p.9]{La} or \cite[Chapter 7]{MR}).
\begin{Lemma}[Six exponentials theorem]\label{LemSET}
Let $\beta_1$ and $\beta_2$
be linearly independent complex numbers over $\Q$
and let $z_1$, $z_2$ and $z_3$
be linearly independent complex numbers over $\Q$.
Then at least one of the six numbers $e^{\beta_jz_k}$ $(j=1,2; k=1,2,3)$
is transcendental.
\end{Lemma}
\begin{proof}[Proof of Proposition \ref{Propnumber}]
First of all we show $\mathcal{E}_p\cap\Q=\emptyset$ for any prime numbers $p$.
For this purpose
we assume that there exists $\varepsilon\in\mathcal{E}_p\cap\Q$.
From the assumption we write $\varepsilon=\lambda_p(n)=\alpha/\beta$
for some $n\in\Z_{\geq 1}$ and $\alpha$, $\beta\in\Z_{>0}$.
This implies
\begin{equation}\label{eq:number1}
p^{(\alpha+\beta)\kappa}(p^{\kappa n}-1)^\beta=(p^{\kappa(n+1)}-1)^\beta.
\end{equation}
We write $\kappa=a/b$ with $a$, $b\in\Z_{>0}$.
Let $K=\Q(p^{1/b})$ and $\mathcal{O}_K$ be its integer ring.
Then the minimal polynomial of $p^{1/b}$ over $\Q$
is $X^b-p$ thanks to Eisenstein's criterion.
This implies $[K:\Q]=b$ and $p^{1/b}\in\mathcal{O}_K$.
We see from $[K:\Q]=b$ that the $\Q$-embeddings from $K$ into $\C$
are given by $p^{1/b}\mapsto p^{1/b}e(j/b)$ for $j=0,1,\ldots,b-1$,
where $e(x)=e^{2\pi ix}$.
Thus we have
\[
N_{K/\Q}(p^{1/b})=\prod_{j=0}^{b-1}(p^{1/b}e(j/b))
=(-1)^{b-1}p\notin\Z^{\times},
\]
which implies $p^{1/b}\notin\mathcal{O}_K^{\times}$.
Thus there exists a maximal ideal $\mathfrak{m}$ of $\mathcal{O}_K$
such that $p^{1/b}\in\mathfrak{m}$.
Then by (\ref{eq:number1}) we find $p^{\kappa(n+1)}-1\in\mathfrak{m}$.
This is a contradiction because $p^{\kappa(n+1)}\in\mathfrak{m}$.
Therefore we conclude that $\mathcal{E}_p\cap\Q=\emptyset$.

We are ready to prove Proposition \ref{Propnumber}.
In view of Proposition \ref{PropkappaCA} it is sufficient to show
$\mathcal{E}_p\cap\mathcal{E}_q\cap\mathcal{E}_r=\emptyset$
for distinct prime numbers $p$, $q$ and $r$.
Assume that
$\mathcal{E}_p\cap\mathcal{E}_q\cap\mathcal{E}_r\neq\emptyset$.
We take
$\varepsilon
\in\mathcal{E}_p\cap\mathcal{E}_q\cap\mathcal{E}_r$
and write $\varepsilon=\lambda_p(l)=\lambda_q(m)=\lambda_r(n)$
with $l$, $m$, $n\in\Z_{\geq 1}$.
Then we see from the above discussion that $\varepsilon\notin\Q$ holds.
We apply the six exponentials theorem with $\beta_1=1$, $\beta_2=\varepsilon$,
$z_1=\log(p^{\kappa})$, $z_2=\log(q^{\kappa})$, $z_3=\log(r^{\kappa})$.
In consequence at least one of
\[
p^{\kappa},~~~~q^{\kappa},~~~~r^{\kappa},~~~~\frac{p^{\kappa(l+1)}-1}{p^{\kappa}(p^{\kappa l}-1)},~~~~
\frac{q^{\kappa(m+1)}-1}{q^{\kappa}(q^{\kappa m}-1)},~~~~
\frac{r^{\kappa(n+1)}-1}{r^{\kappa}(r^{\kappa n}-1)}
\]
is transcendental.
However all of these numbers are algebraic numbers, which is a contradiction.
This completes the proof.
\end{proof}
We mention the four exponentials conjecture, which is stated as follows:
\begin{Conjecture}[Four exponentials conjecture]
Let $\beta_1$ and $\beta_2$ be linearly independent complex numbers over $\Q$
and let $z_1$ and $z_2$ be linearly independent complex numbers over $\Q$.
Then at least one of the four numbers $e^{\beta_jz_k}$ $(j=1,2; k=1,2)$
is transcendental.
\end{Conjecture}
Replacing the six exponentials theorem
with the four exponentials conjecture
in the proof of Proposition \ref{Propnumber}, we easily see
that the four exponentials conjecture implies
$\mathcal{E}_p\cap\mathcal{E}_q=\emptyset$
for distinct prime numbers $p$ and $q$.
Consequently we obtain
\begin{Proposition}\label{PropCFEC}
Let $\kappa\in\Q_{>0}$ and $\varepsilon\in\R_{>0}$.
Then the four exponentials conjecture implies that
the number of $\kappa$-SHCNs with parameter $\varepsilon$
is $1$ or $2$.
\end{Proposition}
\begin{Remark}\label{RemgenSCHN}
 In contrast to Proposition \ref{PropCFEC},
the number of $\kappa$-SHCNs with fixed parameter $\varepsilon>0$
is not necessarily $1$ or $2$ for general $\kappa\in\R_{>0}$.
To see this, we note that $\lambda_p(m)=\lambda_q(n)$ if and only if
\[
\log q\times
\log\left(\frac{p^{\kappa(m+1)}-1}{p^{\kappa}(p^{\kappa m}-1)}\right)
=\log p\times
\log\left(\frac{q^{\kappa(n+1)}-1}{q^{\kappa}(q^{\kappa n}-1)}\right).
\]
We take positive integers $m$, $n$ and distinct prime numbers $p$, $q$
and fix them.
We set
\[
f(\kappa)=
\log q\times
\log\left(\frac{p^{\kappa(m+1)}-1}{p^{\kappa}(p^{\kappa m}-1)}\right)
-\log p\times
\log\left(\frac{q^{\kappa(n+1)}-1}{q^{\kappa}(q^{\kappa n}-1)}\right).
\]
We easily see that
\begin{equation}\label{eq:lim1}
\lim_{\kappa\downarrow 0}f(\kappa)
=\log q\times\log\left(\frac{m+1}{m}\right)
-\log p\times\log\left(\frac{n+1}{n}\right).
\end{equation}
On the other hand, we have
\begin{equation}\label{eq:lim2}
f(\kappa)=p^{-\kappa m}\log q-q^{-\kappa n}\log p
+O_{p,q,m,n}(p^{-\kappa(m+1)})+O_{p,q,m,n}(q^{-\kappa(n+1)})
\end{equation}
as $\kappa\to\infty$.
It follows from (\ref{eq:lim2}) that $f(\kappa)>0$ for sufficiently
large $\kappa$ if $p^m<q^n$.
We note that there exist $p$, $q$, $m$, $n$ such that
$p^m<q^n$ holds and (\ref{eq:lim1}) is negative.
In fact, it is easy to check that such conditions are satisfied
when $p=7$, $m=1$, $q=2$, $n=3$.
Under these conditions there is $\kappa_0\in(0,\infty)$
such that $f(\kappa_0)=0$ by the intermediate value theorem.
Consequently, we see that at least four $\kappa_0$-SHCNs with parameter $\lambda_p(m)$
exist.
\end{Remark}
\begin{Remark}
See Waldschmidt's survey \cite{Wa}
for the four exponentials conjecture, the six exponentials theorem
and related topics.
\end{Remark}
\section{Behavior of $\sigma_{\kappa}(n)$ on $\kappa$-SHCNs, I}\label{Sec:sigSHCN1}
Next we investigate the size of $N_j$ as $j\to\infty$,
or equivalently, the size of
$N(\varepsilon)$ as $\varepsilon\downarrow 0$.
For this purpose we set
\[
 F(x,k)=\left.\log\left(\frac{x^{k+1}-1}{x(x^{k}-1)}\right)\right/\log x
\]
for $x>1$ and $k\in\Z_{\geq 1}$.
We notice that $\lambda_p(k)=F(p^{\kappa},k)$.
Since
\begin{equation}\label{eq:ftF}
 F(x,k)=\left.\log\left(1+\frac{1}{x^k+x^{k-1}+\cdots+x}\right)\right/\log x,
\end{equation}
the function $F(x,k)$ is monotonically decreasing in $x\in(1,\infty)$
in the strict sense for each $k\in\Z_{\geq 1}$.
We also note that
$F(x,k)\to \infty$ as $x\downarrow 1$ and $F(x,k)\to 0$
as $x\to \infty$ for each $k\in\Z_{\geq 1}$.
Thus,
for any $k\in\Z_{\geq 1}$ and $\varepsilon\in(0,\infty)$
there exists a unique $x_k(\varepsilon)\in(1,\infty)$ such that
\begin{equation}\label{eq:defxk}
 F(x_k(\varepsilon),k)=\varepsilon.
\end{equation}
It is easy to see that $x_k(\varepsilon)\to \infty$ as
$\varepsilon\downarrow 0$.
The monotonicity of $F(\cdot,k)$ also says that
for any $k\in\Z_{\geq 1}$ and $x>1$
there exists a unique $y_k(x)\in(1,\infty)$ satisfying
\begin{equation}\label{eq:defyk}
 F(y_k(x),k)=F(x,1).
\end{equation}
By definition we have $y_1(x)=x$.
We collect elementary or known facts for $y_k(x)$:
\begin{Lemma}\label{Lemyk}
The function $y_k(x)$ has the following properties:
 \begin{enumerate}
  \item We have $y_k(x)>y_{k+1}(x)$ for any $k\in\Z_{\geq 1}$ and $x\in(1,\infty)$.
  \item We have $x^{1/k}<y_k(x)<(kx)^{1/k}$ for any $k\in\Z_{\geq 2}$ and
$x\in(1,\infty)$.
In particular, $y_k(x)$ tends to $1$ as $k\to \infty$ for each $x\in(1,\infty)$.
  \item Let $k\in\Z_{\geq 1}$ be fixed. Then as $x\to\infty$ we have
\[
 y_k(x)=(kx)^{1/k}\left(1+O_k\left(\frac{1}{\log x}\right)\right).
\]
 \end{enumerate}
\end{Lemma}
\begin{proof}
 We show the claim (1).
By (\ref{eq:ftF}) we notice $F(y,k)>F(y,k+1)$ for any $y\in(1,\infty)$.
Thus we have $F(y_k(x),k)=F(x,1)=F(y_{k+1}(x),k+1)<F(y_{k+1}(x),k)$
for any $x\in(1,\infty)$.
Since $F(\cdot,k)$ is strictly decreasing, we obtain $y_k(x)>y_{k+1}(x)$.

For the claim (2) we can find $x^{1/k}<y_k(x)$ in \cite[p.236]{Ro1}
or \cite[p.190]{Ro2} and $y_k(x)<(kx)^{1/k}$ in \cite[Lemma 1 in \S3]{CNS}.

We show the claim (3).
We simply write $y=y_k(x)$.
Then we have $F(y,k)=F(x,1)$.
Since $y$ diverges to $\infty$ as $x\to \infty$, we have
\begin{align*}
 F(y,k)&=\frac{1}{\log y}
\left(\frac{1}{y^k+y^{k-1}+\cdots+y}+O_k\left(\frac{1}{y^{2k}}\right)\right)\\
&=\frac{1}{y^k\log y}\left(1+O_k\left(\frac{1}{y}\right)\right).
\end{align*}
In the same manner we find
\[
 F(x,1)=\frac{1}{x\log x}\left(1+O\left(\frac{1}{x}\right)\right).
\]
Combining these with the claim (2), we obtain
\begin{equation}\label{eq:asympy}
 y^k\log y=x\log x\left(1+O_k(x^{-1/k})\right).
\end{equation}
We see from the claim (2) that
$\log y=k^{-1}\log x(1+O_k((\log x)^{-1}))$.
Applying this to (\ref{eq:asympy}), we have
$y^k=kx(1+O_k((\log x)^{-1}))$.
Taking the $1/k$-th power, we obtain the claim (3).
\end{proof}
Let $\kappa$ be a fixed positive real number.
For any $\varepsilon>0$
we define $N=N(\varepsilon)$ by (\ref{eq:Neps}).
We also put $x=x_1(\varepsilon)$.
Then we have
\begin{Proposition}\label{PropEP}
 Keep the notation as above.
Let $L$ be a fixed positive integer satisfying
$\kappa(L+1)>1$.
Then as $\varepsilon\downarrow 0$
it holds that
\begin{align*}
 \log\frac{\sigma_{\kappa}(N)}{N^{\kappa}}
=&\sum_{p\leq x^{1/\kappa}}\sum_{k=1}^{\infty}\frac{p^{-k\kappa}}{k}
-\sum_{l=1}^{L-1}\sum_{y_{l+1}(x)^{1/\kappa}<p\leq y_l(x)^{1/\kappa}}
\sum_{k=1}^{\infty}\frac{p^{-k(l+1)\kappa}}{k}\\
&+O_{\kappa,L}\left(\frac{1}{x^{1-\frac{1}{\kappa(L+1)}}\log x}\right).
\end{align*}
\end{Proposition}
\begin{proof}
Since $\sigma_{\kappa}(n)$ is multiplicative, we have
\begin{equation}\label{eq:sigmaCA1}
 \frac{\sigma_{\kappa}(N)}{N^{\kappa}}
=\prod_p\frac{1-p^{-\kappa(\lfloor\lambda_p(\varepsilon)\rfloor+1)}}{1-p^{-\kappa}}
=
\prod_{l=0}^{\infty}
\prod_{\begin{subarray}{c}
           p\\
          l\leq\lambda_p(\varepsilon)<l+1
       \end{subarray}}
\frac{1-p^{-\kappa(l+1)}}{1-p^{-\kappa}}.
\end{equation}
We see from Lemma \ref{Lemlambda}, $\lambda_p(k)=F(p^{\kappa},k)$
and the monotonicity of $F(\cdot,k)$ for each $k\in\Z_{\geq 1}$
that $l\leq\lambda_p(\varepsilon)<l+1$ is equivalent to
$y_{l+1}(x)^{1/\kappa}<p\leq y_l(x)^{1/\kappa}$.
Thus (\ref{eq:sigmaCA1}) is
\begin{align*}
\frac{\sigma_{\kappa}(N)}{N^{\kappa}}&=\prod_{l=1}^{\infty}
\prod_{y_{l+1}(x)^{1/\kappa}<p\leq y_l(x)^{1/\kappa}}
\frac{1-p^{-\kappa(l+1)}}{1-p^{-\kappa}}\\
&=\prod_{p\leq x^{1/\kappa}}(1-p^{-\kappa})^{-1}
\times\prod_{l=1}^{\infty}\prod_{y_{l+1}(x)^{1/\kappa}<p\leq y_l(x)^{1/\kappa}}
(1-p^{-\kappa(l+1)}).
\end{align*}
Taking the logarithm and applying the Taylor expansion for $\log(1-z)$
at $z=0$, we have
\begin{equation}\label{eq:sigmaCA2}
 \log\frac{\sigma_{\kappa}(N)}{N^{\kappa}}
=\sum_{p\leq x^{1/\kappa}}\sum_{k=1}^{\infty}\frac{p^{-k\kappa}}{k}
-\sum_{l=1}^{\infty}\sum_{y_{l+1}(x)^{1/\kappa}<p\leq y_l(x)^{1/\kappa}}
\sum_{k=1}^{\infty}\frac{p^{-k(l+1)\kappa}}{k}.
\end{equation}
We divide the range of $l$ into $1\leq l\leq L-1$, $l=L$ and $l\geq L+1$.
We treat $l=L$.
We keep the assumption $\kappa(L+1)>1$ in mind.
Then by the prime number theorem and Lemma \ref{Lemyk}
we obtain
\begin{equation}\label{eq:sigmaCA3}
 \begin{aligned}
 \sum_{y_{L+1}(x)^{1/\kappa}<p\leq y_L(x)^{1/\kappa}}
\sum_{k=1}^{\infty}\frac{p^{-k(L+1)\kappa}}{k}
&\ll \sum_{p>y_{L+1}(x)^{1/\kappa}}p^{-(L+1)\kappa}\\
&\ll_{\kappa,L}\frac{1}{x^{1-\frac{1}{\kappa(L+1)}}\log x}.
 \end{aligned}
\end{equation}
Next we consider $l\geq L+1$ in (\ref{eq:sigmaCA2}).
If $y_{l+1}(x)^{1/\kappa}<p\leq y_l(x)^{1/\kappa}$,
then $p^{(l+1)\kappa}>y_{l+1}(x)^{l+1}>x$ holds
thanks to Lemma \ref{Lemyk} (2).
Thus we obtain
 \begin{align*}
  \sum_{l=L+1}^{\infty}\sum_{y_{l+1}(x)^{1/\kappa}<p\leq y_l(x)^{1/\kappa}}
\sum_{k=1}^{\infty}\frac{p^{-k(l+1)\kappa}}{k}
&\ll\sum_{l=L+1}^{\infty}\sum_{y_{l+1}(x)^{1/\kappa}<p\leq y_l(x)^{1/\kappa}}
p^{-(l+1)\kappa}\\
&\leq x^{-1}\pi(y_{L+1}(x)^{1/\kappa}).
\end{align*}
Applying the prime number theorem and Lemma \ref{Lemyk} (2) or (3),
we see that this is
\begin{equation}\label{eq:sigmaCA4}
 \ll_{\kappa,L} \frac{1}{x^{1-\frac{1}{\kappa(L+1)}}\log x}.
\end{equation}
Inserting (\ref{eq:sigmaCA3}) and (\ref{eq:sigmaCA4}) into (\ref{eq:sigmaCA2}),
we reach the result.
\end{proof}
In order to compare $x$ with $N$, we show
\begin{Lemma}\label{Lem:Mtheta}
Retain the notation. Then it holds that
\[
 \log N=\vartheta(x^{1/\kappa})+O_{\kappa}(x^{1/(2\kappa)})
\]
as $\varepsilon\downarrow 0$.
\end{Lemma}
\begin{proof}
Take any small $\varepsilon>0$ satisfying $x\geq\max\{e^5,e^{5/\kappa}\}$.
 In the same manner as the proof of Proposition \ref{PropEP}
we have
\begin{equation}\label{eq:M1}
 N=\prod_{l=1}^{\infty}
\prod_{y_{l+1}(x)^{1/\kappa}<p\leq y_{l}(x)^{1/\kappa}}
p^l.
\end{equation}
By Lemma \ref{Lemyk} (2) we have
 \begin{equation}\label{eq:ykesti}
 y_l(x)^{1/\kappa}<2,
\end{equation}
provided $l\geq (\log x)^2$.
Thus the product over $p$ in (\ref{eq:M1}) equals $1$ if $l\geq (\log x)^2$.
Taking the logarithm on (\ref{eq:M1}), we have
\begin{equation}\label{eq:M2}
\begin{aligned}
\log N&=\sum_{l=1}^{\infty}l
\sum_{y_{l+1}(x)^{1/\kappa}<p\leq y_l(x)^{1/\kappa}}\log p\\
&=\sum_{l=1}^{\infty}
l(\vartheta(y_l(x)^{1/\kappa})-\vartheta(y_{l+1}(x)^{1/\kappa}))\\
&=\sum_{l=1}^{\infty}\vartheta(y_l(x)^{1/\kappa}).
\end{aligned}
\end{equation}
Thanks to (\ref{eq:ykesti}), the summand $\vartheta(y_l(x)^{1/\kappa})$
vanishes if $l\geq(\log x)^2$.
By the prime number theorem and Lemma \ref{Lemyk} we have
$\vartheta(y_2(x)^{1/\kappa}) \ll_{\kappa} x^{1/(2\kappa)}.$
We see from Lemma \ref{Lemyk} and
$\vartheta(y)\ll y$ that
\[
 \sum_{3\leq l<(\log x)^2}\vartheta(y_{l}(x)^{1/\kappa})
\leq\vartheta(y_3(x)^{1/\kappa})(\log x)^2
\ll x^{1/(3\kappa)}(\log x)^2.
\]
Applying these to (\ref{eq:M2}),
we complete the proof.
\end{proof}
\section{The Euler product on the critical strip}\label{Sec:EP}
In view of Proposition \ref{PropEP} we would like to know
the behavior of the partial Euler product for $\zeta(\kappa)$
for $\kappa\in(0,1)$.
So we develop the behavior of the Euler product in this section.
We note that some results can be found in \cite{Ak}.
However many of the results in \cite{Ak} are conditional.
In order to discuss the behavior of the Euler product unconditionally
as far as possible,
we treat it again by using other methods.
While we treat the case $\kappa$ is complex in \cite{Ak},
we restrict our attention to the case when $\kappa$ is real here.

To start with, we write the logarithm of the partial Euler product
in terms of
the von Mangoldt function $\Lambda(n)$.
\begin{Lemma}\label{LemEPMangoldt}
 For $\kappa>0$ and $X>1$ we have
\[
\sum_{p\leq X}\sum_{k=1}^{\infty}\frac{p^{-k\kappa}}{k}
=\sum_{k=1}^{\infty}\sum_{l=1}^{\infty}\frac{\mu(k)}{kl}
\sum_{2\leq n\leq X^{1/k}}\frac{\Lambda(n)}{n^{kl\kappa}\log n},
\]
where $\mu(n)$ is the M\"{o}bius function.
\end{Lemma}
We remark that
the sum over $n$ vanishes if $k>\log X/\log 2$.
This guarantees the absolute convergence of the sums on the right-hand side.
\begin{proof}
We consider
\[
 \sum_{k=1}^{\infty}\frac{\mu(k)}{k}
\sum_{2\leq n\leq X^{1/k}}\frac{\Lambda(n)}{n^{k\kappa}\log n}.
\]
Since the support of $\Lambda(n)$ consists of prime powers,
this is
\[
=\sum_{k=1}^{\infty}\frac{\mu(k)}{k}
\sum_{
\begin{subarray}{c}
p,l\\
p^l\leq X^{1/k}
\end{subarray}}
\frac{p^{-kl\kappa}}{l}
=\sum_{k=1}^{\infty}\sum_{l=1}^{\infty}\frac{\mu(k)}{kl}
\sum_{p\leq X^{1/(kl)}}p^{-kl\kappa}.
\]
Replacing $l$ by $m=kl$, we see that this equals
\[
 =\sum_{m=1}^{\infty}
\left(\sum_{k\mid m}\mu(k)\right)\frac{1}{m}\sum_{p\leq X^{1/m}}p^{-m\kappa}.
\]
We recall the following fundamental formula for the M\"{o}bius function:
\begin{equation}\label{eq:Moebius0}
 \sum_{k\mid m}\mu(k)=\begin{cases}
		       1 & \text{if }m=1,\\
		       0 & \text{if }m>1.
		      \end{cases}
\end{equation}
Applying this, we obtain
\[
 \sum_{k=1}^{\infty}\frac{\mu(k)}{k}
\sum_{2\leq n\leq X^{1/k}}\frac{\Lambda(n)}{n^{k\kappa}\log n}
=\sum_{p\leq X}p^{-\kappa}.
\]
Replacing $\kappa$ by $l\kappa$, multiplying $1/l$
and summing it over $l\geq 1$,
we reach the stated result.
\end{proof}
A truncated version of Lemma \ref{LemEPMangoldt} is given as follows:
\begin{Lemma}\label{LemEPMangoldt2}
Let $\kappa$ be a fixed positive real number and
$L$ be a given positive integer satisfying $\kappa(L+1)>1$.
Then as $X\to\infty$ we have
\[
\sum_{p\leq X}\sum_{k=1}^{\infty}\frac{p^{-k\kappa}}{k}
=\sum_{\begin{subarray}{c}
          k,l\geq 1\\
          kl\leq L
       \end{subarray}}\frac{\mu(k)}{kl}
\sum_{2\leq n\leq X^{1/k}}\frac{\Lambda(n)}{n^{kl\kappa}\log n}
+O_{\kappa,L}\left(\frac{1}{X^{\kappa-\frac{1}{L+1}}(\log X)^2}\right).
\]
\end{Lemma}
\begin{proof}
 In view of Lemma \ref{LemEPMangoldt} it suffices to show
\begin{equation}\label{eq:EPVM0}
 \sum_{\begin{subarray}{c}
          k,l\geq 1\\
          kl>L
       \end{subarray}}\frac{\mu(k)}{kl}
\sum_{2\leq n\leq X^{1/k}}\frac{\Lambda(n)}{n^{kl\kappa}\log n}
=O_{\kappa,L}\left(\frac{1}{X^{\kappa-\frac{1}{L+1}}(\log X)^2}\right).
\end{equation}
Replacing $l$ by $c=kl$ and changing the order of the sums, we have
\[
\sum_{\begin{subarray}{c}
          k,l\geq 1\\
          kl>L
       \end{subarray}}\frac{\mu(k)}{kl}
\sum_{2\leq n\leq X^{1/k}}\frac{\Lambda(n)}{n^{kl\kappa}\log n}
=\sum_{c=L+1}^{\infty}\frac{1}{c}
\sum_{2\leq n\leq X}
\left(\sum_{\begin{subarray}{c}
                  k\mid c\\
                  k\leq\log X/\log n
	    \end{subarray}}
\mu(k)
\right)
\frac{\Lambda(n)}{n^{c\kappa}\log n}.
\]
Thanks to (\ref{eq:Moebius0}), the parenthesis vanishes
if $c\leq \log X/\log n$, equivalently, $n\leq X^{1/c}$.
When $n>X^{1/c}$, the trivial estimate gives that
the absolute value of the parenthesis is bounded by
$\leq\log X/\log n$.
Consequently we obtain
\begin{equation}\label{eq:EPVM1}
 \left|
\sum_{\begin{subarray}{c}
          k,l\geq 1\\
          kl>L
       \end{subarray}}\frac{\mu(k)}{kl}
\sum_{2\leq n\leq X^{1/k}}\frac{\Lambda(n)}{n^{kl\kappa}\log n}
\right|
\leq\log X\sum_{c=L+1}^{\infty}\frac{1}{c}
\sum_{X^{1/c}<n\leq X}\frac{\Lambda(n)}{n^{c\kappa}(\log n)^2}.
\end{equation}
We treat the sum over $n$ on the right.
When $c>\log X/\log 2$, the sum over $n$ is $\ll_{\kappa} 2^{-c\kappa}$.
Thus we have
\begin{equation}\label{eq:EPVM2}
 \sum_{c>\log X/\log 2}\frac{1}{c}
\sum_{X^{1/c}<n\leq X}\frac{\Lambda(n)}{n^{c\kappa}(\log n)^2}
\ll_{\kappa}\sum_{c>\log X/\log 2}\frac{2^{-c\kappa}}{c}
\ll_{\kappa}\frac{1}{X^{\kappa}\log X}.
\end{equation}
On the other hand, when $L+1\leq c\leq \log X/\log 2$,
integration by part and $\psi(u)=\sum_{n\leq u}\Lambda(n)\ll u$ give
\begin{align*}
 \sum_{X^{1/c}<n\leq X}\frac{\Lambda(n)}{n^{c\kappa}(\log n)^2}
&\leq\left(\frac{c}{\log X}\right)^2\sum_{n>X^{1/c}}\frac{\Lambda(n)}{n^{c\kappa}}
=\left(\frac{c}{\log X}\right)^2\int_{X^{1/c}}^{\infty}\frac{d\psi(u)}{u^{c\kappa}}\\
&\leq \left(\frac{c}{\log X}\right)^2c\kappa\int_{X^{1/c}}^{\infty}\frac{\psi(u)}{u^{c\kappa+1}}du\\
&\ll_{\kappa,L}\left(\frac{c}{\log X}\right)^2\frac{1}{X^{\kappa-\frac{1}{c}}}.
\end{align*}
Thus, dividing $L+1\leq c\leq \log X/\log2$ into
$c=L+1$ and $L+1<c\leq\log X/\log 2$,
we obtain
\begin{equation}\label{eq:EPVM3}
\begin{aligned}
&\sum_{L+1\leq c<\log X/\log 2}\frac{1}{c}
\sum_{X^{1/c}<n\leq X}\frac{\Lambda(n)}{n^{c\kappa}(\log n)^2}\\
&\ll_{\kappa,L}\frac{1}{X^{\kappa-\frac{1}{L+1}}(\log X)^2}
+\frac{1}{X^{\kappa-\frac{1}{L+2}}(\log X)^2}\times(\log X)^2
\ll_{\kappa,L} \frac{1}{X^{\kappa-\frac{1}{L+1}}(\log X)^2}.
\end{aligned}
\end{equation}
Inserting (\ref{eq:EPVM2}) and (\ref{eq:EPVM3}) into (\ref{eq:EPVM1}),
we obtain (\ref{eq:EPVM0}) as claimed.
\end{proof}
Next we establish the following explicit formula for $(\zeta'/\zeta)(s)$.
\begin{Proposition}
 Let $X\geq 1$ and $s=\sigma+it\in\C$ satisfying $\zeta(s)\neq 0$, $\infty$.
Then we have
\begin{equation}\label{eqexp1}
\begin{aligned}
&-\sum_{n\leq X}\Lambda(n)n^{-s}+\frac{X^{1-s}}{1-s}
+\frac{\psi(X)-X}{X^{s}}\\
&=\frac{\zeta'}{\zeta}(s)-X^{-s}\frac{\zeta'}{\zeta}(0)
+\sum_{\rho}\left(\frac{1}{\rho-s}-\frac{1}{\rho}\right)X^{\rho-s}
+\sum_{k=1}^{\infty}
\left(
\frac{1}{2k}-\frac{1}{2k+s}
\right)
X^{-2k-s},
\end{aligned}
\end{equation}
where $\rho$ runs over the nontrivial zeros of $\zeta(s)$
counted with multiplicity.
\end{Proposition}
\begin{proof}
We consider
\begin{equation}\label{eqINT}
 \frac{1}{2\pi i}
\int_{c-i\infty}^{c+i\infty}
\frac{\zeta'}{\zeta}(w)\left(\frac{1}{w-s}-\frac{1}{w}\right)X^wdw,
\end{equation}
where $c=\max\{2,1+\sigma\}$.
It is easy to see that the integral converges absolutely.
Since $c\geq 2$,
we have $(\zeta'/\zeta)(w)=-\sum_{n=1}^{\infty}\Lambda(n)n^{-w}$
on $\Rep(w)=c$.
Thus (\ref{eqINT}) is
\[
 =-\sum_{n=1}^{\infty}\Lambda(n)
\frac{1}{2\pi i
}\int_{c-i\infty}^{c+i\infty}\left(\frac{1}{w-s}-\frac{1}{w}\right)
\left(\frac{X}{n}\right)^wdw.
\]
Here we can easily justify the interchange of the sum and the integral
by checking the absolute convergence.
Applying Perron's formula, we obtain
\begin{equation}\label{eqPrimeSum}
 \frac{1}{2\pi i}
\int_{c-i\infty}^{c+i\infty}
\frac{\zeta'}{\zeta}(w)\left(\frac{1}{w-s}-\frac{1}{w}\right)X^wdw
=-X^s\sum_{n\leq X}\Lambda(n)n^{-s}+\psi(X).
\end{equation}

Next we calculate (\ref{eqINT}) by the residue theorem.
We take $K\in\Z_{\geq 1}$ satisfying $-(2K+1)<\sigma$ and temporally fix $K$.
According to \cite[Lemmas 12.2 and 12.4 in pp.398--399]{MV},
for any $n\in\Z_{\geq 2}$
there exists $T_n\in[n,n+1]$ such that
$(\zeta'/\zeta)(u+ iT_n)\ll_{\sigma,K}(\log n)^2$ holds
uniformly for $-(2K+1)\leq u\leq c$.
For $n\in\Z_{\geq 2}$ satisfying $n>|t|$,
we take the rectangle $\mathcal{C}_{K,n}$ with vertices at
$c-iT_n$, $c+iT_n$, $-(2K+1)+iT_n$ and $-(2K+1)-iT_n$.
Then the residue theorem says
\begin{align*}
&\frac{1}{2\pi i}
\int_{\mathcal{C}_{K,n}}
\frac{\zeta'}{\zeta}(w)\left(\frac{1}{w-s}-\frac{1}{w}\right)X^wdw\\
&=
\Res_{w=s}+\Res_{w=0}+\Res_{w=1}
+\sum_{\begin{subarray}{c}
           \rho=\beta+i\gamma\\
           -T_n<\gamma<T_n
       \end{subarray}}
\Res_{w=\rho}
+\sum_{k=1}^K\Res_{w=-2k}\\
&=X^s\frac{\zeta'}{\zeta}(s)-\frac{\zeta'}{\zeta}(0)
-\left(\frac{1}{1-s}-1\right)X
+\sum_{\begin{subarray}{c}
          \rho=\beta+i\gamma\\
          -T_n<\gamma<T_n
       \end{subarray}}
\left(\frac{1}{\rho-s}-\frac{1}{\rho}\right)X^{\rho}\nonumber\\
&+\sum_{k=1}^{K}
\left(
\frac{1}{2k}-\frac{1}{2k+s}
\right)
X^{-2k}.
\end{align*}
It follows from the choice of $T_n$
and
$\overline{(\zeta'/\zeta)(\overline{w})}=(\zeta'/\zeta)(w)$
that the integral on the horizontal lines tends to $0$ as $n\to\infty$.
Thus we have
\begin{align*}
&\frac{1}{2\pi i}
\left(
\int_{c-i\infty}^{c+i\infty}-\int_{-(2K+1)-i\infty}^{-(2K+1)+i\infty}
\right)
\frac{\zeta'}{\zeta}(w)\left(\frac{1}{w-s}-\frac{1}{w}\right)X^wdw\\
&=X^s\frac{\zeta'}{\zeta}(s)-\frac{\zeta'}{\zeta}(0)
-\left(\frac{1}{1-s}-1\right)X
+\sum_{\rho}
\left(\frac{1}{\rho-s}-\frac{1}{\rho}\right)X^{\rho}\nonumber\\
&+\sum_{k=1}^{K}
\left(
\frac{1}{2k}-\frac{1}{2k+s}
\right)
X^{-2k}.
\end{align*}
It is easy to check the absolute convergence of the
sums and integrals in the above formula.
Next we take the limit $K\to\infty$.
Since $(\zeta'/\zeta)(-(2K+1)+iv)\ll \log(K+|v|)$ holds for $v\in\R$
(see \cite[Lemma 12.4]{MV}),
the integral over $\Rep(w)=-(2K+1)$
goes to $0$ as $K\to\infty$.
Combining it with (\ref{eqPrimeSum}), we reach the stated result.
\end{proof}
Here we recall the exponential integral.
For $z=x+iy\in\C\setminus[0,\infty)$ the exponential integral
$\Ei(z)$ is defined by
\[
 \Ei(z)=\lim_{R\to\infty}\int_{-R+iy}^z\frac{e^t}{t}dt,
\]
where the integral path is the straight line joining $-R+iy$ and $z$.
The function $\Ei(z)$ is a single-valued holomorphic function
in $\C\setminus[0,\infty)$, and has a logarithmic branch
at $z=0$: see \cite[eq.~(3.1.6) in p.32]{Le}.
With this notation we have
\begin{Proposition}\label{PropEP2}
Let $\kappa\geq 0$ and $X>1$.
Put
\begin{equation}\label{eq:sumzeros}
\begin{gathered}
Z(\kappa;X)
=-\sum_{\rho}
\left(
\Ei((\rho-\kappa)\log X)-\frac{X^{\rho-\kappa}}{\rho\log X}
\right),\\
 R(\kappa;X)=\frac{\zeta'}{\zeta}(0)\frac{1}{X^{\kappa}\log X}
-\sum_{k=1}^{\infty}\left(
\frac{X^{-2k-\kappa}}{2k\log X}
+\Ei(-(\kappa+2k)\log X)\right).
\end{gathered}
\end{equation}
Then the following claims hold:
\begin{enumerate}
 \item When $0\leq\kappa<1$, we have
\begin{equation}\label{eq:EP2-1}
\sum_{2\leq n\leq X}\frac{\Lambda(n)}{n^{\kappa}\log n}
-\Li(X^{1-\kappa})-\frac{\psi(X)-X}{X^{\kappa}\log X}
=\log(-\zeta(\kappa))
+Z(\kappa;X)+R(\kappa;X).
\end{equation}
 \item When $\kappa=1$, we have
\begin{equation}\label{eq:EP2-2}
\sum_{2\leq n\leq X}\frac{\Lambda(n)}{n\log n}
-\frac{\psi(X)-X}{X\log X}
=\log\log X+\EC
+Z(1;X)+R(1;X).
\end{equation}
\item When $\kappa>1$, we have
\begin{equation}\label{eq:EP2-3}
\begin{aligned}
&\sum_{2\leq n\leq X}\frac{\Lambda(n)}{n^{\kappa}\log n}
-\Ei(-(\kappa-1)\log X)-\frac{\psi(X)-X}{X^{\kappa}\log X}\\
&=\log\zeta(\kappa)
+Z(\kappa;X)+R(\kappa;X).
\end{aligned}
\end{equation}
\end{enumerate}
\end{Proposition}
\begin{proof}
Let $\kappa\geq 0$.
We take
\begin{equation}\label{eqint}
 \frac{1}{2}
\left(
    \int_{\infty+i0}^{\kappa+i0}+\int_{\infty-i0}^{\kappa-i0}
\right)\cdots ds
\end{equation}
on (\ref{eqexp1}).\footnote{%
In the case $\kappa=1$
we move the term $X^{1-s}/(1-s)$ to the right-hand side
before we take the integral.
}
It is easy to see that
\begin{equation}\label{eqint1}
\begin{gathered}
  \frac{1}{2}
\left(
    \int_{\infty+i0}^{\kappa+i0}+\int_{\infty-i0}^{\kappa-i0}
\right)\left(-\sum_{n\leq X}\Lambda(n)n^{-s}\right)ds
=\sum_{2\leq n\leq X}\frac{\Lambda(n)}{n^{\kappa}\log n},\\
  \frac{1}{2}
\left(
    \int_{\infty+i0}^{\kappa+i0}+\int_{\infty-i0}^{\kappa-i0}
\right)
\frac{\psi(X)-X}{X^s}ds
=-\frac{\psi(X)-X}{X^{\kappa}\log X},\\
  \frac{1}{2}
\left(
    \int_{\infty+i0}^{\kappa+i0}+\int_{\infty-i0}^{\kappa-i0}
\right)
\left(-X^{-s}\frac{\zeta'}{\zeta}(0)\right)ds
=\frac{\zeta'}{\zeta}(0)\frac{1}{X^{\kappa}\log X}.
\end{gathered}
\end{equation}
We deal with the sum over the nontrivial zeros on (\ref{eqexp1}).
We have
\begin{equation}\label{eqint5}
\begin{aligned}
&\frac{1}{2}
\left(
    \int_{\infty+i0}^{\kappa+i0}+\int_{\infty-i0}^{\kappa-i0}
\right)
\sum_{\rho}
\left(
  \frac{1}{\rho-s}-\frac{1}{\rho}
\right)
X^{\rho-s}ds\\
&=\int_{\infty}^{\kappa}
\sum_{\rho}
\left(
  \frac{1}{\rho-s}-\frac{1}{\rho}
\right)
X^{\rho-s}ds\\
&=-\sum_{\rho}
\left(
\Ei((\rho-\kappa)\log X)-\frac{X^{\rho-\kappa}}{\rho\log X}
\right)
=Z(\kappa;X).
\end{aligned}
\end{equation}
Here in the second equality we interchanged the sum and the integral
and replaced $s$ by $w=(\rho-s)\log X$.
We can justify the interchange of the sum and integral by
checking the absolute convergence,
using $\sum_{\rho=\beta+i\gamma}|\gamma|^{-2}<\infty$.
In the same manner we have
\begin{equation}\label{eqint6}
\begin{aligned}
&\frac{1}{2}
\left(
    \int_{\infty+i0}^{\kappa+i0}+\int_{\infty-i0}^{\kappa-i0}
\right)
\sum_{k=1}^{\infty}
\left(
  \frac{1}{2k}-\frac{1}{2k+s}
\right)
X^{-2k-s}
ds\\
&=-\sum_{k=1}^{\infty}
\left(
  \frac{X^{-2k-\kappa}}{2k\log X}+\Ei(-(\kappa+2k)\log X)
\right).
 \end{aligned}
\end{equation}

We consider the case $0\leq \kappa<1$.
We treat the second term on the left-hand side of (\ref{eqexp1}).
Changing the variable $w=(1-s)\log x$, we have
\[
 \int_{\infty\pm i0}^{\kappa\pm i0}\frac{X^{1-s}}{1-s}ds
=-\Ei((1-\kappa)\log X\mp i0).
\]
Applying \cite[eq.~(3.4.7) and Problem 8 in \S 3]{Le},\footnote{%
In \cite{Le} the logarithmic integral (\ref{eq:defLi})
is denoted by $\Li_1$ instead of $\Li$.
}
we have
\begin{equation}\label{eqint3}
  \frac{1}{2}
\left(
    \int_{\infty+i0}^{\kappa+i0}+\int_{\infty-i0}^{\kappa-i0}
\right)\frac{X^{1-s}}{1-s}ds
=-\Li(X^{1-\kappa}).
\end{equation}
Next we treat the first term on the right-hand side of (\ref{eqexp1}).
We write
\begin{align*}
&\frac{1}{2}
\left(
    \int_{\infty+i0}^{\kappa+i0}+\int_{\infty-i0}^{\kappa-i0}
\right)
\frac{\zeta'}{\zeta}(s)ds\\
&=\int_{\infty}^{2}\frac{\zeta'}{\zeta}(s)ds
+\int_{2}^{\kappa}\left(\frac{\zeta'}{\zeta}(s)+\frac{1}{s-1}\right)ds
-\frac{1}{2}
\left(
    \int_{2+i0}^{\kappa+i0}+\int_{2-i0}^{\kappa-i0}
\right)
\frac{1}{s-1}ds\\
&=I_1+I_2+I_3.
\end{align*}
We easily see $I_1=\log\zeta(2)$ and
$I_2=\log((\kappa-1)\zeta(\kappa))-\log\zeta(2)$.
We compute $I_3$.
We choose the logarithmic branch of $\log(s-1)$ by
continuous variation along the straight lines joining $s=2\pm i0$
and $s=\kappa\pm i0$,
with $\log(s-1)|_{s=2}=0$.
Then we have $\log(s-1)|_{s=\kappa\pm i0}=\log(1-\kappa)\pm\pi i$.
This yields $I_3=-\log(1-\kappa)$.
In consequence we obtain
\begin{equation}\label{eqint4}
 \frac{1}{2}
\left(
    \int_{\infty+i0}^{\kappa+i0}+\int_{\infty-i0}^{\kappa-i0}
\right)
\frac{\zeta'}{\zeta}(s)ds
=\log(-\zeta(\kappa)).
\end{equation}
Applying (\ref{eqint1})--(\ref{eqint4}) to (\ref{eqint}),
we reach (\ref{eq:EP2-1}).

Next we deal with the case $\kappa=1$.
We calculate
\begin{equation}\label{eqint2-1}
 \int_{\infty}^1\left(\frac{\zeta'}{\zeta}(s)+\frac{X^{1-s}}{s-1}\right)ds.
\end{equation}
We divide the integral into $\int_{\infty}^2+\int_2^1$.
We deal with the first integral.
For the integral of $X^{1-s}/(s-1)$ we replace $s$ by $(s-1)\log X=v$,
so that
\begin{equation}\label{eqint2-2}
 \int_{\infty}^2\left(\frac{\zeta'}{\zeta}(s)+\frac{X^{1-s}}{s-1}\right)ds
=\log\zeta(2)-\int_{\log X}^{\infty}\frac{e^{-v}}{v}dv.
\end{equation}
We calculate the remaining integral.
We write it as
\begin{equation}\label{eqint2-3}
 \int_2^1\left(\frac{\zeta'}{\zeta}(s)+\frac{X^{1-s}}{s-1}\right)ds
=\int_2^1\left(\frac{\zeta'}{\zeta}(s)+\frac{1}{s-1}\right)ds
+\int_2^1\frac{X^{1-s}-1}{s-1}ds.
\end{equation}
Since $(s-1)\zeta(s)$ tends to $1$ as $s\to 1$, the former integral
equals $-\log\zeta(2)$.
For the latter integral we replace $s$ by $(s-1)\log X=v$.
In consequence (\ref{eqint2-3}) equals
\begin{align*}
 &=-\log\zeta(2)+\int_{0}^{\log X}\frac{1-e^{-v}}{v}dv\\
&=-\log\zeta(2)+\int_0^{1}\frac{1-e^{-v}}{v}dv
+\log\log X-\int_1^{\log X}\frac{e^{-v}}{v}dv.
\end{align*}
We insert this and (\ref{eqint2-2}) into (\ref{eqint2-1}).
Using \cite[p.8]{Le}, we reach
\begin{align*}
  \int_{\infty}^1\left(\frac{\zeta'}{\zeta}(s)+\frac{X^{1-s}}{s-1}\right)ds
&=\log\log X+\int_0^1\frac{1-e^{-v}}{v}dv
-\int_1^{\infty}\frac{e^{-v}}{v}dv\\
&=\log\log X+\EC.
\end{align*}
This and
(\ref{eqint1})--(\ref{eqint6}) give (\ref{eq:EP2-2}).

Finally we treat the case $\kappa>1$.
We easily see that
\[
  \int_{\infty}^{\kappa}\frac{X^{1-s}}{1-s}ds
=-\Ei(-(\kappa-1)\log X),\qquad
 \int_{\infty}^{\kappa}\frac{\zeta'}{\zeta}(s)ds
=\log\zeta(\kappa).
\]
These together with (\ref{eqint1})--(\ref{eqint6})
imply (\ref{eq:EP2-3}).
\end{proof}
Finally in this section we give bounds for $Z(\kappa;X)$ and $R(\kappa;X)$.
We set
\[
 \Theta=\sup\{\Rep(\rho):\rho\in\C,~\zeta(\rho)=0\}.
\]
We know $1/2\leq\Theta\leq 1$ and that the Riemann hypothesis
is equivalent to $\Theta=1/2$.
We show
\begin{Lemma}\label{Lem:zerosumest}
 Let $\kappa\in(0,1)$. Then for $X\geq 2$ we have
\[
 Z(\kappa;X)=O_{\kappa}\left(\frac{X^{\Theta-\kappa}}{\log X}\right),\qquad
R(\kappa;X)=O_{\kappa}\left(\frac{1}{X^{\kappa}\log X}\right).
\]
\end{Lemma}
\begin{proof}
 We recall that integration by parts gives
\[
 \Ei(z)=\frac{e^z}{z}+O_{\delta}\left(\frac{e^{\Rep(z)}}{|z|^2}\right)
\]
on $\{z\in\C:|z|\geq 1,~|\arg(-z)|\leq\pi-\delta\}$ for each $\delta>0$
(see \cite[eq.~(3.2.4) in p.33]{Le}).
We apply this to (\ref{eq:sumzeros}). Taking $\sum_{\rho}|\rho|^{-2}<\infty$
into account, we reach the claimed estimates.
\end{proof}
\section{Behavior of $\sigma_{\kappa}(n)$ on $\kappa$-SHCNs, II}\label{Sec:sigSHCN2}
In this section we rewrite Proposition \ref{PropEP}
in the case $1/2\leq\kappa<1$.
We note that results in this section are unconditional,
so that we can use them not depending on the validity of the assumption
(i) in Theorem \ref{Thm1}.

Let $\kappa\in[1/2,1)$.
For $\varepsilon>0$
we continue to use the notation $N=N(\varepsilon)$ and $x=x_1(\varepsilon)$
as (\ref{eq:Neps}) and (\ref{eq:defxk}), respectively.
Then Proposition \ref{PropEP} turns to
\begin{Lemma}\label{Lem:divonSHCN1}
Keep the notation as above.
Then we have
\begin{equation}\label{eq:divonSHCN1}
 \log\frac{\sigma_{\kappa}(N)}{N^{\kappa}}
=\sum_{2\leq n\leq x^{1/\kappa}}\frac{\Lambda(n)}{n^{\kappa}\log n}
+A_{\kappa}+O_{\kappa}\left(\frac{1}{x^{1-\frac{1}{2\kappa}}\log x}\right),
\end{equation}
where
\[
 A_{\kappa}=
\begin{cases}
 -(\log 2)/2 & \text{if $\kappa=1/2$,}\\
 0 & \text{if $1/2<\kappa<1$.}
\end{cases}
\]
\end{Lemma}
\begin{proof}
Firstly we consider the case $1/2<\kappa<1$.
We specialize $L=1$ in Proposition \ref{PropEP}.
Applying Lemma \ref{LemEPMangoldt2} with
$L=1$ and $X=x^{1/\kappa}$, we reach the result.

Next we consider the case $\kappa=1/2$.
Proposition \ref{PropEP} with $L=2$
is read as
\begin{equation}\label{eq:divonSHCN3}
 \log\frac{\sigma_{1/2}(N)}{N^{1/2}}
=\sum_{p\leq x^2}\sum_{k=1}^{\infty}
\frac{p^{-k/2}}{k}
-\sum_{y_2(x)^2<p\leq x^2}\sum_{k=1}^{\infty}\frac{p^{-k}}{k}
+O\left(\frac{1}{x^{1/3}\log x}\right).
\end{equation}
We compute the second term on the right.
By the Mertens theorem (see \cite[eq.~(2.16) in p.51]{MV})
and Lemma \ref{Lemyk} (3)
we have
\begin{equation}\label{eq:divonSHCN4}
 -\sum_{y_2(x)^2<p\leq x^2}\sum_{k=1}^{\infty}\frac{p^{-k}}{k}
=\log\frac{\log y_2(x)}{\log x}+O\left(\frac{1}{\log y_2(x)}\right)
=-\log 2+O\left(\frac{1}{\log x}\right).
\end{equation}
We treat the first term on the right-hand side of
(\ref{eq:divonSHCN3}).
We start with Lemma \ref{LemEPMangoldt2} with $\kappa=1/2$, $L=2$
and $X=x^2$. We apply the Mertens theorem (see \cite[the first displayed formula in p.52]{MV}), so that
\begin{equation}\label{eq:divonSHCN5}
\begin{aligned}
&\sum_{p\leq x^2}\sum_{k=1}^{\infty}
\frac{p^{-k/2}}{k}\\
&=\sum_{2\leq n\leq x^2}\frac{\Lambda(n)}{n^{1/2}\log n}
-\frac{1}{2}\sum_{2\leq n\leq x}\frac{\Lambda(n)}{n\log n}
+\frac{1}{2}\sum_{2\leq n\leq x^2}\frac{\Lambda(n)}{n\log n}
+O\left(\frac{1}{x^{1/3}(\log x)^2}\right)\\
&=\sum_{2\leq n\leq x^2}\frac{\Lambda(n)}{n^{1/2}\log n}
+\frac{1}{2}\log 2+O\left(\frac{1}{\log x}\right).
\end{aligned}
\end{equation}
Inserting (\ref{eq:divonSHCN4}) and (\ref{eq:divonSHCN5})
into (\ref{eq:divonSHCN3}), we obtain the result.
\end{proof}
\section{Proof of the implication {\rm (i)}$\Longrightarrow${\rm (iii)}
in Theorem \ref{Thm1}}\label{Sec:pf1}
In this section we prove the implication (i)$\Longrightarrow$(iii)
in Theorem \ref{Thm1}.
For this purpose we firstly give asymptotic formulas for
$\sigma_{\kappa}(n)$ on $\kappa$-SHCNs
under assumptions on zero-free regions for $\zeta(s)$.
To state it, let $\kappa\in[1/2,1)$.
For $j\in\Z_{\geq 1}$ we determine $\varepsilon_j$ and $N_j$
as in Proposition \ref{PropkappaCA}.
Then the asymptotic formulas can be stated as follows:
\begin{Proposition}\label{Prop:divSHCNlim}
Keep the notation as above.
We assume $\zeta(s)\neq 0$ in $\Rep(s)>\kappa$.
Then we have
\begin{equation}\label{eq:divSHCNlim1}
 \lim_{j\to\infty}a_{\kappa}(N_j)
=-B_{\kappa}\zeta(\kappa),
\end{equation}
where
\[
 B_{\kappa}=e^{A_{\kappa}}
=\begin{cases}
  1/\sqrt{2} & \text{if $\kappa=1/2$,}\\
     1       & \text{if $1/2<\kappa<1$.}
 \end{cases}
\]
\end{Proposition}
For $\varepsilon>0$ we continue to use the notations $N=N(\varepsilon)$
and $x=x_1(\varepsilon)$.
Then it suffices to show
that $a_{\kappa}(N)$ tends to $-B_{\kappa}\zeta(\kappa)$
as $\varepsilon\downarrow 0$.
To show this,
we insert (\ref{eq:EP2-1}) with $X=x^{1/\kappa}$
into Lemma \ref{Lem:divonSHCN1}.
Apply $\psi(x^{1/\kappa})=\vartheta(x^{1/\kappa})+O(x^{1/(2\kappa)})$
and Lemma \ref{Lem:zerosumest},
we obtain
\begin{equation}\label{eq:sigmaM1}
\begin{aligned}
 \log\frac{\sigma_{\kappa}(N)}{N^{\kappa}}
&=\Li(x^{(1-\kappa)/\kappa})
+\frac{\vartheta(x^{1/\kappa})-x^{1/\kappa}}{x\log(x^{1/\kappa})}
+\log(-\zeta(\kappa))+A_{\kappa}\\
&+O_{\kappa}\left(\frac{1}{x^{(\kappa-\Theta)/\kappa}\log x}\right)
+O_{\kappa}\left(\frac{1}{x^{1-\frac{1}{2\kappa}}\log x}\right).
\end{aligned}
\end{equation}
We note that the assumption $\zeta(s)\neq 0$ in $\Rep(s)>\kappa$
implies $\Theta\leq\kappa$.
Thus the $O$-terms tend to $0$ as $\varepsilon\downarrow 0$.
To express the first two terms on the right
in terms of $N$, we shall show
\begin{Lemma}\label{Lem:Liest1}
 Let $\kappa\in(0,1)$. We suppose that the parameter $H\in\R$
satisfies $H=o(X)$ as $X\to\infty$.
Then we have
\[
 \Li((X+H)^{1-\kappa})-\Li(X^{1-\kappa})
=\frac{H}{X^{\kappa}\log X}
+O_{\kappa}\left(\frac{|H|^2}{X^{\kappa+1}\log X}\right)
\]
as $X\to\infty$.
\end{Lemma}
\begin{proof}
We start with the definition (\ref{eq:defLi}).
Changing the variable $u\mapsto (u+X)^{1-\kappa}$,
we notice
\begin{equation}\label{eq:Liest1-1}
\Li((X+H)^{1-\kappa})-\Li(X^{1-\kappa})
=\int_0^H\frac{du}{(u+X)^{\kappa}\log(u+X)}.
\end{equation}
The Taylor expansions for $(1+Y)^{-\kappa}$ and $\log(1+Y)$ give
\begin{gather*}
 (u+X)^{-\kappa}=X^{-\kappa}\left(1+O_{\kappa}\left(\frac{|u|}{X}\right)\right),\\
\frac{1}{\log(u+X)}=\frac{1}{\log X}
\left(1+O\left(\frac{|u|}{X\log X}\right)\right)
\end{gather*}
as $X\to\infty$ uniformly on $|u|\leq|H|$.
Applying these to (\ref{eq:Liest1-1}),
we obtain the claimed formula.
\end{proof}
\begin{proof}[Proof of Proposition \ref{Prop:divSHCNlim}]
We specialize $X=x^{1/\kappa}$
and $H=\log N-x^{1/\kappa}$ in Lemma \ref{Lem:Liest1}.
By Lemma \ref{Lem:Mtheta} we notice
$H=\vartheta(x^{1/\kappa})-x^{1/\kappa}+O(x^{1/(2\kappa)})
(=o(x^{1/\kappa}))$.
Thus we have
\begin{equation}\label{eq:Liest2}
\begin{aligned}
&\Li((\log N)^{1-\kappa})-\Li(x^{(1-\kappa)/\kappa})\\
&=\frac{\vartheta(x^{1/\kappa})-x^{1/\kappa}}{x\log(x^{1/\kappa})}
+O_{\kappa}\left(\frac{1}{x^{1-\frac{1}{2\kappa}}\log x}\right)
+O_{\kappa}\left(\frac{|\vartheta(x^{1/\kappa})-x^{1/\kappa}|^2}{x^{1+\frac{1}{\kappa}}\log x}\right).
\end{aligned}
\end{equation}
The assumption that $\zeta(s)$ does not vanish in $\Rep(s)>\kappa$
implies $\vartheta(X)=X+O_{\kappa}(X^{\kappa}(\log X)^2)$,
so that the last $O$-term on (\ref{eq:Liest2}) is
$O_{\kappa}((\log x)^3/x^{\frac{1}{\kappa}-1})$.
Since $1/2\leq\kappa<1$,
the $O$-terms on (\ref{eq:Liest2}) tend to $0$
as $\varepsilon\downarrow 0$.

Applying (\ref{eq:Liest2}) to (\ref{eq:sigmaM1})
and taking the exponential, we conclude that
$a_{\kappa}(N)$ tends to
$-B_{\kappa}\zeta(\kappa)$ as $\varepsilon\downarrow 0$.
We complete the proof.
\end{proof}
\begin{proof}[Proof of the implication {\rm (i)}$\Longrightarrow${\rm (iii)} in Theorem \ref{Thm1}]
Let $\kappa\in[1/2,1)$ and we assume $\zeta(s)\neq 0$ in $\Rep(s)>\kappa$.
Then we shall show
\begin{equation}\label{eq:pf1to3-0}
\limsup_{n\to\infty}a_{\kappa}(n)=-B_{\kappa}\zeta(\kappa).
\end{equation}
Proposition \ref{Prop:divSHCNlim} implies
(LHS)$\geq$(RHS).
Below we prove the opposite inequality.
We take arbitrary $\delta>0$. Then by Proposition \ref{Prop:divSHCNlim}
there exists $J\in\Z_{\geq 3}$ such that
\begin{equation}\label{eq:pf1to3-1}
 j\geq J\Longrightarrow a_{\kappa}(N_j)\leq -B_{\kappa}\zeta(\kappa)+\delta.
\end{equation}
We take any $n\in\Z$ satisfying $n\geq N_J$.
Then there exists $j\geq J$ such that $N_j\leq n<N_{j+1}$.
By Proposition \ref{PropkappaCA} (4), both $N_j$ and $N_{j+1}$
are $\kappa$-SHCNs with parameter $\varepsilon_j$.
This implies
\begin{equation}\label{eq:pf1to3-2}
 a_{\kappa}(n)\leq\left(\frac{n}{N}\right)^{\kappa\varepsilon_j}
\frac{\exp[\Li((\log N)^{1-\kappa})]}{\exp[\Li((\log n)^{1-\kappa})]}
a_{\kappa}(N)
\end{equation}
for each $N\in\{N_j,N_{j+1}\}$.
For $u\geq 2$ we set $f(u)=\kappa\varepsilon_j u-\Li(u^{1-\kappa})$.
Then we find
\[
 f'(u)=\kappa\varepsilon_j-\frac{1}{u^{\kappa}\log u},\phantom{MM}
f''(u)=\frac{\kappa\log u+1}{u^{\kappa+1}(\log u)^2}.
\]
Thus we notice $f''(u)>0$ on $u\geq 2$, so that
the function $f$ is convex on $u\in[2,\infty)$.
Since $N_j\leq n<N_{j+1}$, we can write $\log n=t\log N_j+(1-t)\log N_{j+1}$
for some $t\in[0,1]$.
Then the convexity of $f$ gives
\begin{align*}
 f(\log n)&\leq t f(\log N_j)+(1-t)f(\log N_{j+1})\\
&\leq\max\{f(\log N_j),f(\log N_{j+1})\}.
\end{align*}
Thus there exists $N\in\{N_j,N_{j+1}\}$ such that $f(\log n)\leq f(\log N)$.
Taking the exponential, we obtain
\[
 \left(\frac{n}{N}\right)^{\kappa\varepsilon_j}
\frac{\exp[\Li((\log N)^{1-\kappa})]}{\exp[\Li((\log n)^{1-\kappa})]}
\leq 1.
\]
Inserting this into (\ref{eq:pf1to3-2}) and applying (\ref{eq:pf1to3-1}),
we reach $a_{\kappa}(n)\leq-B_{\kappa}\zeta(\kappa)+\delta$.
This implies (LHS)$\leq$(RHS) on (\ref{eq:pf1to3-0}).
This completes the proof.
\end{proof}
\section{Proof of the implication {\rm (ii)}$\Longrightarrow${\rm (i)}
in Theorem \ref{Thm1}}\label{Sec:pf2}
In this section we prove the implication (ii)$\Longrightarrow$(i)
in Theorem \ref{Thm1}.
Let $\kappa\in[1/2,1)$.
In order to prove the implication,
we show that $a_{\kappa}(n)$ is not bounded if
there is a zero $\rho$ of $\zeta(s)$ with $\Rep(\rho)>\kappa$.
First of all we give the following estimates for (\ref{eq:EP2-1}):
\begin{Proposition}\label{Prop:EPomega}
Let $\kappa\in[1/2,1)$. Assume that there is a zero of
$\zeta(s)$ in $\Rep(s)>\kappa$.
Then we have
\begin{equation}\label{eq:EPomega}
 \sum_{2\leq n\leq X}\frac{\Lambda(n)}{n^{\kappa}\log n}
-\Li(X^{1-\kappa})-\frac{\psi(X)-X}{X^{\kappa}\log X}
=\Omega_{+}(X^{\Theta-\kappa-\delta})
\end{equation}
as $X\to\infty$ for each $\delta>0$.
\end{Proposition}
To show this, we recall an analog of Landau's theorem.
Let $F:[2,\infty)\to \R$ be a bounded Riemann integrable function
on any finite interval in $[2,\infty)$.
We also suppose $F(u)\geq 0$ on $u\geq X_0$ for some
$X_0\in[2,\infty)$.
We define the \emph{abscissa of convergence} by
\[
 \ac=\inf\left\{
\sigma\in\R:\int_{X_0}^{\infty}F(u)u^{-\sigma}du<\infty
\right\}.
\]
We easily see that the integral
\begin{equation}\label{eq:IntTrans}
 \wF(s)=\int_2^{\infty}F(u)u^{-s}du
\end{equation}
converges absolutely and uniformly on $\Rep(s)\geq\ac+\delta$
for each $\delta>0$,
so that $\widetilde{F}(s)$ is holomorphic in $\Rep(s)>\ac$.
Landau's theorem for (\ref{eq:IntTrans}) can be stated as follows
(see \cite[Lemma 15.1 in p.463]{MV}):
\begin{Lemma}\label{Lem:Landau}
 Suppose that a function $F:[2,\infty)\to\R$ satisfies the above
conditions. Then $\widetilde{F}(s)$ cannot be holomorphically
continued to a region including $s=\ac$.
\end{Lemma}
We prove Proposition \ref{Prop:EPomega}.
By the assumption we notice $\Theta>\kappa$.
We take any $\delta\in(0,\Theta-\kappa)$.
Set
\begin{equation}\label{eq:Fplus}
 F_{+}(u)=\frac{1}{u^{1-\kappa}}
\left(
  u^{\Theta-\kappa-\delta}
 -\sum_{2\leq n\leq u}\frac{\Lambda(n)}{n^{\kappa}\log n}
 +\Li(u^{1-\kappa})+\frac{\psi(u)-u}{u^{\kappa}\log u}
\right).
\end{equation}
Below we compute $\wFp(s)$.
\begin{Lemma}\label{Lem:FplusIT}
 For $s\in\R_{>1}$ we have
\begin{equation}\label{eq:FplusIT}
\begin{aligned}
 \wFp(s)&=\frac{2^{-s+\Theta-\delta}}{s-\Theta+\delta}
-\frac{1}{s-\kappa}\log((s-1)\zeta(s))\\
&-\int_s^{\infty}
\left(
\frac{1}{z}\frac{\zeta'}{\zeta}(z)+\frac{2^{1-z}}{z-1}
\right)dz
+r(s),
\end{aligned}
\end{equation}
where $r(s)$ is a holomorphic function in $\Rep(s)>\kappa$.
\end{Lemma}
\begin{proof}
We easily find
\begin{equation}\label{eq:FplusIT1}
 \int_2^{\infty}u^{-s+\Theta-1-\delta}du
=\frac{2^{-s+\Theta-\delta}}{s-\Theta+\delta}.
\end{equation}
Next we deal with the integral transform of the second term
on (\ref{eq:Fplus}).
We notice
\[
 \int_{\uparrow 2}^{\infty}
u^{-s+\kappa}
d\left(
  \sum_{2\leq n\leq u}\frac{\Lambda(n)}{n^{\kappa}\log n}
\right)
=\sum_{n=2}^{\infty}\frac{\Lambda(n)}{n^s\log n}=\log\zeta(s).
\]
We apply integration by parts.
Since $\sum_{2\leq n\leq u}\Lambda(n)/(n^{\kappa}\log n)\leq\sum_{n\leq u}n^{-\kappa}\ll_{\kappa} u^{1-\kappa}$ on $u\geq 2$, we obtain
\begin{equation}\label{eq:FplusIT2}
 \int_2^{\infty}u^{-s+\kappa-1}
\left(
  \sum_{2\leq n\leq u}\frac{\Lambda(n)}{n^{\kappa}\log n}
\right)du
=\frac{1}{s-\kappa}\log\zeta(s).
\end{equation}
Next we treat the transform of the third term.
Integration by parts gives
\[
 \int_2^{\infty}u^{-s+\kappa-1}\Li(u^{1-\kappa})du
=\frac{2^{-s+\kappa}}{s-\kappa}\Li(2^{1-\kappa})
+\frac{1}{s-\kappa}\int_2^{\infty}\frac{du}{u^s\log u}.
\]
Changing the variable $u$ by $u=e^{t/(s-1)}$ for the last integral,
we have
\begin{align*}
\int_2^{\infty}\frac{du}{u^s\log u}
&=\int_{(s-1)\log 2}^{\infty}\frac{e^{-t}}{t}dt\\
&=\int_1^{\infty}\frac{e^{-t}}{t}dt
-\int_{(s-1)\log 2}^1 \frac{1-e^{-t}}{t}dt
-\log\log 2-\log(s-1).
\end{align*}
We easily see that the first three terms on the last line
are entire functions.
In consequence we obtain
\begin{equation}\label{eq:FplusIT3}
 \int_2^{\infty}u^{-s+\kappa-1}\Li(u^{1-\kappa})du
=-\frac{1}{s-\kappa}\log(s-1)+r(s),
\end{equation}
where $r_+(s)$ is a holomorphic function in $\Rep(s)>\kappa$.
Finally we deal with the transform of the last term
on (\ref{eq:Fplus}).
We notice
\[
 \int_s^{\infty}u^{-z-1}dz=\frac{u^{-s-1}}{\log u}
\]
for $u\geq 2$.
Thus we have
\[
 \int_2^{\infty}\frac{u^{-s-1}}{\log u}(\psi(u)-u)du
=\int_s^{\infty}\int_2^{\infty}\frac{\psi(u)-u}{u^{z+1}}dudz.
\]
Here we interchanged the order of the integration, which
can be justified by checking that $u^{-z}$ is integrable
on $(u,z)\in[2,\infty)\times[s,\infty)$.
We easily compute the integral on $u$, so that
\begin{equation}\label{eq:FplusIT4}
 \int_2^{\infty}\frac{u^{-s-1}}{\log u}(\psi(u)-u)du
=-\int_s^{\infty}
\left(\frac{1}{z}\frac{\zeta'}{\zeta}(z)+\frac{2^{1-z}}{z-1}\right)dz.
\end{equation}
Combining the equations (\ref{eq:FplusIT1})--(\ref{eq:FplusIT4}),
we reach the claimed formula (\ref{eq:FplusIT}).
\end{proof}
Next we investigate analyticity of $\wFp(s)$:
\begin{Lemma}\label{Lem:analwFp}
The following claims hold:
 \begin{enumerate}
  \item The function $\wFp(s)$ is a single-valued holomorphic
function in 
\[
 \{s\in\C:\Rep(s)>\kappa\}\setminus(\{\rho-\lambda:\zeta(\rho)=0,~\lambda\geq 0\}\cup\{\Theta-\delta\}).
\]
In addition, $s=\Theta-\delta$ is a simple pole of $\wFp(s)$
with residue $1$.
\item Let $\rho=\beta+i\gamma$ be a nontrivial zero of $\zeta(s)$
with $\beta>\kappa$.
Suppose $\zeta(\sigma+i\gamma)\neq 0$ for $\sigma>\beta$.
Then we have
\begin{equation}\label{eq:analwFp}
 \wFp(\sigma+i\gamma)=\frac{\kappa m_{\rho}}{\rho(\rho-\kappa)}\log\frac{1}{\sigma-\beta}+O_{\kappa,\delta,\rho}(1)
\end{equation}
as $\sigma\downarrow\beta$, where $m_{\rho}$ is the multiplicity
of the zero of $\zeta(s)$ at $s=\rho$.
In particular, $\wFp(s)$ cannot be holomorphically continued to
a region including $s=\rho$.
 \end{enumerate}
\end{Lemma}
\begin{proof}
 We can easily check the claim (1) by (\ref{eq:FplusIT}).

We prove (2). Set $s=\sigma+i\gamma$.
Since the first and last terms on (\ref{eq:FplusIT}) are holomorphic
at $s=\rho$,
they are $O_{\kappa,\delta}(1)$ as $\sigma\downarrow \beta$.
Next we treat the second term. We notice
\[
 \log((s-1)\zeta(s))-\log\zeta(2)
=\left(\int_2^{2+i\gamma}+\int_{2+i\gamma}^{s}\right)
\left(\frac{\zeta'}{\zeta}(z)+\frac{1}{z-1}\right)dz,
\]
where $\log\zeta(2)\in\R$. Estimating the integrals trivially,
we find
\begin{align*}
 \log((s-1)\zeta(s))
&=\int_{2+i\gamma}^{s}\frac{\zeta'}{\zeta}(z)dz+O_{\rho}(1)\\
&=\int_{2+i\gamma}^s\left(\frac{\zeta'}{\zeta}(z)-\frac{m_{\rho}}{z-\rho}\right)dz+m_{\rho}\int_{2+i\gamma}^s\frac{dz}{z-\rho}+O_{\rho}(1)\\
&=m_{\rho}\log(\sigma-\beta)+O_{\rho}(1).
\end{align*}
Thus we obtain
\[
 -\frac{1}{s-\kappa}\log((s-1)\zeta(s))
=\frac{m_{\rho}}{\rho-\kappa}\log\frac{1}{\sigma-\beta}+O_{\kappa,\rho}(1).
\]
Next we treat the third term. In a similar manner as above
we have
\[
 -\int_s^{\infty}\left(\frac{1}{z}\frac{\zeta'}{\zeta}(z)-\frac{2^{1-z}}{z-1}\right)dz
=-\int_{s}^{2+i\gamma}\frac{m_{\rho}}{z(z-\rho)}dz+O_{\rho}(1).
\]
Writing the integrand as partial fractions, we see that this is
\[
=-\frac{m_{\rho}}{\rho}\int_s^{2+i\gamma}\left(\frac{1}{z-\rho}-\frac{1}{z}\right)dz+O_{\rho}(1)
=-\frac{m_{\rho}}{\rho}\log\frac{1}{\sigma-\beta}+O_{\rho}(1).
\]
Combining the above discussion,
we reach the claimed formula (\ref{eq:analwFp}).
\end{proof}
We are ready to prove Proposition \ref{Prop:EPomega}.
\begin{proof}[Proof of Proposition \ref{Prop:EPomega}]
 Suppose that there exists $X_0=X_0(\delta)\in[2,\infty)$
satisfying
\[
 \sum_{2\leq n\leq u}\frac{\Lambda(n)}{n^{\kappa}\log n}
-\Li(u^{1-\kappa})-\frac{\psi(u)-u}{u^{\kappa}\log u}\leq u^{\Theta-\kappa-\delta}
\] 
for any $u\geq X_0$.
Then $F_+(u)\geq 0$ holds for any $u\geq X_0$.
We see from Lemmas \ref{Lem:Landau} and \ref{Lem:analwFp} that
the abscissa of convergence for $\int_2^{\infty}F_+(u)u^{-s}du$
is $\Theta-\delta$.
This gives that $\wFp(s)$ is holomorphic in $\Rep(s)>\Theta-\delta$.
On the other hand, by the definition of $\Theta$ there exists
a nontrivial zero $\rho=\beta+i\gamma$ of $\zeta(s)$
with $\Theta-\frac{\delta}{2}\leq\beta\leq\Theta$ satisfying
$\zeta(\sigma+i\gamma)\neq 0$ for any $\sigma>\beta$.
By Lemma \ref{Lem:analwFp} (2), the function $\wFp(s)$
is not holomorphic at $s=\rho$.
We have a contradiction.
Thus we conclude the claimed estimate (\ref{eq:EPomega}).
\end{proof}
We prove the implication (ii)$\Longrightarrow$(i) in Theorem \ref{Thm1}.
Let $\kappa\in[1/2,1)$.
We assume that $\zeta(s)$ has a zero in $\Rep(s)>\kappa$.
Under this assumption
we will show that $a_{\kappa}(n)$ is not bounded.

We start with Lemma \ref{Lem:divonSHCN1}.
We see from Lemma \ref{Lem:Mtheta} and $\vartheta(x^{1/\kappa})=\psi(x^{1/\kappa})+O(x^{1/(2\kappa)})$ that $\log N=\psi(x^{1/\kappa})+O(x^{1/(2\kappa)})$,
so that
\begin{equation}\label{eq:Thm1_2to1-1}
 \begin{aligned}
\log a_{\kappa}(N)
&=\left(-\Li((\log N)^{1-\kappa})+\Li(x^{(1-\kappa)/\kappa})+\frac{\log N-x^{1/\kappa}}{x\log(x^{1/\kappa})}\right)\\
&+
\left(
  \sum_{2\leq n\leq x^{1/\kappa}}\frac{\Lambda(n)}{n^{\kappa}\log n}
 -\Li(x^{(1-\kappa)/\kappa})-\frac{\psi(x^{1/\kappa})-x^{1/\kappa}}{x\log(x^{1/\kappa})}
\right)+O_{\kappa}(1).
 \end{aligned}
\end{equation}
By Proposition \ref{Prop:EPomega}, the contribution of the second parentheses 
is $\Omega_{+}(x^{(\Theta-\kappa-\delta)/\kappa})$ for each $\delta>0$.
In particular, it is not bounded above.
We treat the first parentheses.
We remark that Lemma \ref{Lem:Liest1} does not work.
In fact, Lemma \ref{Lem:Liest1} and
$\psi(x^{1/\kappa})-x^{1/\kappa}\ll x^{\Theta/\kappa}(\log x)^2$
give that the contribution of the first parentheses is
$O(x^{(2\Theta-1-\kappa)/\kappa}(\log x)^3)$.
Because of the possibility $\Theta=1$, this
may be larger than $x^{(\Theta-\kappa-\delta)/\kappa}$,
which has appeared in the above $\Omega_{+}$-estimate.
For this reason we use the following inequality instead of
Lemma \ref{Lem:Liest1}:
\begin{Lemma}\label{Lem:Liineq}
 Let $\kappa\in(0,1)$ and $X$, $Y\geq 2$. Then we have
\[
 \Li(Y^{1-\kappa})-\Li(X^{1-\kappa})\leq\frac{Y-X}{X^{\kappa}\log X}.
\]
\end{Lemma}
\begin{proof}
Changing $u\mapsto u^{1-\kappa}$ on (\ref{eq:defLi}), we find
\[
 \Li(Y^{1-\kappa})-\Li(X^{1-\kappa})=\int_X^Y\frac{du}{u^{\kappa}\log u}.
\]
We consider the case $X\leq Y$. Since the integrand is monotonically
decreasing, it is $\leq 1/(X^{\kappa}\log X)$ on $u\in[X,Y]$.
This leads to the claimed inequality.
We consider the case $Y\leq X$. Then the integrand is $\geq 1/(X^{\kappa}\log X)$ on $u\in[Y,X]$.
Integrating it over $[Y,X]$ and multiplying it by $-1$,
we obtain the claimed inequality.
This completes the proof.
\end{proof}
\begin{proof}[Continuation of the proof of the implication {\rm (ii)}$\Longrightarrow${\rm (i)} in Theorem \ref{Thm1}]
Thanks to Lemma \ref{Lem:Liineq} with $X=x^{1/\kappa}$ and
$Y=\log N$, the contribution of the first parentheses
in (\ref{eq:Thm1_2to1-1}) is nonnegative.
As has been explained above, the contribution of the second parentheses
is not bounded above.
In consequence,
the total contribution of (\ref{eq:Thm1_2to1-1}) is not bounded above.
This completes the proof.
\end{proof}
\section{Proof of Theorem \ref{Thm3}}\label{Sec:pf3}
In this section we prove Theorem \ref{Thm3}.
For this purpose we recall
\begin{Lemma}\label{Lem:pf3-1}
As $X\to\infty$ we have
\[
 \sum_{2\leq n\leq X}\frac{\Lambda(n)}{n^{1/2}\log n}
=\sum_{p\leq X}\sum_{k=1}^{\infty}\frac{p^{-k/2}}{k}
-\frac{1}{2}\log 2+O\left(\frac{1}{\log X}\right).
\]
\end{Lemma}
\begin{proof}
 This is a special case of \cite[Lemma 2.1]{Ak}.
\end{proof}
\begin{proof}[Proof of the implication {\rm (i)}$\Longrightarrow${\rm (iii)}
in Theorem \ref{Thm3}]
 This is \cite[eq.~(359)]{Ra4}. We can also find it in
\cite[Proposition 5.1]{Ak}.
We give another proof here.
We assume the Riemann hypothesis.
We insert Lemmas \ref{Lem:pf3-1} and \ref{Lem:zerosumest}
into (\ref{eq:EP2-1}) with $\kappa=1/2$.
We also apply $\psi(X)=\vartheta(X)+O(X^{1/2})$,
so that
\begin{equation}\label{eq:pf3-1}
 \sum_{p\leq X}\sum_{k=1}^{\infty}
\frac{p^{-k/2}}{k}-\Li(X^{1/2})-\frac{\vartheta(X)-X}{X^{1/2}\log X}
=\frac{1}{2}\log 2+\log(-\zeta(1/2))+O\left(\frac{1}{\log X}\right).
\end{equation}
By Lemma \ref{Lem:Liest1} with $\kappa=1/2$ and $H=\vartheta(X)-X$
we find
\[
 \Li(\vartheta(X)^{1/2})-\Li(X^{1/2})
=\frac{\vartheta(X)-X}{X^{1/2}\log X}
+O\left(\frac{(\log X)^3}{X^{1/2}}\right).
\]
Here we used the fact that the Riemann hypothesis implies
$\vartheta(X)=X+O(X^{1/2}(\log X)^2)$.
Inserting this into (\ref{eq:pf3-1}) and taking the exponential,
we conclude the condition (iii).
\end{proof}
\begin{proof}[Proof of the implication {\rm (ii)}$\Longrightarrow${\rm (i)}
in Theorem \ref{Thm3}]
It suffices to show that $E_1(X)$ is not bounded on $X\geq 3$
if the Riemann hypothesis is false.
Inserting Lemma \ref{Lem:pf3-1} and $\psi(X)=\vartheta(X)+O(X^{1/2})$
into (\ref{eq:EPomega}), we have
\begin{equation}\label{eq:pf3-2}
 \sum_{p\leq X}\sum_{k=1}^{\infty}\frac{p^{-k/2}}{k}
-\Li(X^{1/2})-\frac{\vartheta(X)-X}{X^{1/2}\log X}
=\Omega_+(X^{\Theta-\frac{1}{2}-\delta})
\end{equation} 
as $X\to\infty$ for every $\delta\in(0,\Theta-\frac{1}{2})$.
We write
\begin{equation}\label{eq:pf3-3}
\begin{aligned}
\sum_{p\leq X}\sum_{k=1}^{\infty}\frac{p^{-k/2}}{k}-\Li(\vartheta(X)^{1/2})
&=
\left(
   \sum_{p\leq X}\sum_{k=1}^{\infty}\frac{p^{-k/2}}{k}-\Li(X^{1/2})
   -\frac{\vartheta(X)-X}{X^{1/2}\log X}
\right)\\
&+\left(-\Li(\vartheta(X)^{1/2})+\Li(X^{1/2})+\frac{\vartheta(X)-X}{X^{1/2}\log X}\right).
\end{aligned}
\end{equation}
By (\ref{eq:pf3-2}) the contribution of the first parentheses
is $\Omega_+(X^{\Theta-\frac{1}{2}-\delta})$.
We see from Lemma \ref{Lem:Liineq} with $\kappa=1/2$ and $Y=\vartheta(X)$
that the contribution of the second parentheses is nonnegative.
Thus (\ref{eq:pf3-3}) is not bounded above.
Taking the exponential, we see that $E_1(X)$ is not bounded.
This completes the proof of the implication (ii)$\Longrightarrow$(i).
\end{proof}
\begin{proof}[Proof of Theorem \ref{Thm3}]
We have already shown the implications
(ii)$\Longrightarrow$(i)
and (i)$\Longrightarrow$(iii).
The implication (iii)$\Longrightarrow$(ii) is trivial.
This completes the proof.
\end{proof}

\end{document}